\documentclass[reqno, 11pt, a4paper]{amsart}
\usepackage{latexsym,ifthen,xspace}
\usepackage{amsmath,amssymb,amsthm}
\usepackage{enumerate, calc, bbm}
\usepackage[usenames,dvipsnames]{xcolor}
\usepackage{paralist}
\usepackage[colorlinks=true, pdfstartview=FitV, linkcolor=blue,
  citecolor=blue, urlcolor=blue, pagebackref=false]{hyperref}
\usepackage[text={33pc,605pt},centering]{geometry}    
\usepackage{graphicx}
\usepackage{orcidlink}
\newtheorem{theorem}{Theorem}[section]
\newtheorem{lemma}[theorem]{Lemma}
\newtheorem{prop}[theorem]{Proposition}
\newtheorem{assumption}[theorem]{Assumption}

\theoremstyle{definition}

\theoremstyle{remark}
\newtheorem{remark}[theorem]{Remark}

\numberwithin{equation}{section}

\DeclareMathAlphabet{\mathsl}{OT1}{cmss}{m}{sl}
\SetMathAlphabet{\mathsl}{bold}{OT1}{cmss}{bx}{sl}

\newcommand{\al}{\ensuremath{\alpha}}

\newcommand{\cC}{\ensuremath{\mathcal C}}
\newcommand{\cD}{\ensuremath{\mathcal D}}
\newcommand{\cE}{\ensuremath{\mathcal E}}
\newcommand{\cF}{\ensuremath{\mathcal F}}

\newcommand{\cL}{\ensuremath{\mathcal L}}
\newcommand{\cM}{\ensuremath{\mathcal M}}

\newcommand{\bbR}{\ensuremath{\mathbb R}}

\newcommand{\bbZ}{\ensuremath{\mathbb Z}}
\newcommand{\bfE}{\ensuremath{\mathbf E}}

\newcommand{\bfP}{\ensuremath{\mathbf P}}

\newcommand{\R}{\mathbb R}
\newcommand{\scpr}[3]{%
  \ensuremath{%
    \bigl\langle
      #1, #2
    \bigr\rangle_{\raisebox{-0ex}{$\scriptstyle \ell^{\raisebox{.1ex}{$\scriptscriptstyle 2$}} (#3)$}}
  }
}
\newcommand{\norm}[3]{%
   \ensuremath{%
     \mathchoice{\bigl\lVert #1 \bigr\rVert}
     {\lVert #1 \rVert}
     {\lVert #1 \rVert}
     {\lVert #1 \rVert}_{\raisebox{-.0ex}{$\scriptstyle \ell^{\raisebox{.2ex}{$\scriptscriptstyle #2$}} (#3)$}}
   }
}

\newcommand{\Norm}[2]{%
  \ensuremath{%
    \mathchoice{\bigl\lVert #1 \bigr\rVert}
     {\lVert #1 \rVert}
     {\lVert #1 \rVert}
     {\lVert #1 \rVert}_{\raisebox{-.0ex}{$\scriptstyle #2$}}
  }
}

\DeclareMathOperator{\mean}{\mathbb{E}}
\DeclareMathOperator{\Mean}{\mathrm{E}}
\DeclareMathOperator{\prob}{\mathbb{P}} 
\DeclareMathOperator{\Prob}{\mathrm{P}} 

\DeclareMathOperator{\supp}{\mathrm{supp}}

\newcommand{\av}[1]{\mathop{\mathrm{av}}(#1)}

\newcommand{\ldef}{\ensuremath{\mathrel{\mathop:}=}}
\newcommand{\rdef}{\ensuremath{=\mathrel{\mathop:}}}

\newcommand{\indicator}{\mathbbm{1}}

\begin{document}

\title[Anchored Nash inequalities for RCM with long-range jumps]{Anchored Nash inequalities and heat kernel bounds for a class of random conductance models with long-range jumps}


\author{Sebastian Andres \orcidlink{0000-0002-5116-7789}}
\address{Technische Universit\"at Braunschweig}
\curraddr{Institut f\"ur Mathematische Stochastik, Universit\"atsplatz 2, 38106 Braunschweig}
\email{sebastian.andres@tu-braunschweig.de}
\thanks{}

\author{Xin Chen}
\address{Shanghai Jiao Tong University}
\curraddr{School of Mathematical Sciences, 200240 Shanghai, P.R. China.}
\email{chenxin217@sjtu.edu.cn}
\thanks{}

\author{Martin Slowik \orcidlink{0000-0001-5373-5754}}
\address{University of Mannheim}
\curraddr{Mathematical Institute, B6, 26, 68159 Mannheim}
\email{slowik@math.uni-mannheim.de}
\thanks{}

\author{Kun Yin}
\address{Shanghai Jiao Tong University}
\curraddr{School of Mathematical Sciences, 200240 Shanghai, P.R. China.}
\email{epsilonyk@sjtu.edu.cn}
\thanks{}



\date{\today}

\dedicatory{}

\begin{abstract}
  We show anchored versions of the Nash inequality for discrete non-local divergence-form operators with degenerate weights. They allow to control the $L^{2}$-norm of a function by Dirichlet forms that are not uniformly elliptic. We then use them to provide on-diagonal heat kernel upper bounds for a class of random conductance models with degenerate jump rates allowing long-range jumps. The results are established on a class of graphs including the integer lattice and possibly correlated supercritical percolation clusters.
\end{abstract}

\maketitle

\tableofcontents

\section{Introduction}\label{sec:INTRO}
Nash inequalities, introduced by Nash in his 1958 paper \cite{Na58} (see the first inequality on p.\ 936) to obtain H\"older continuity of solutions for PDEs associated with second order differential operators in divergence form, are closely related to the behaviour of the associated heat kernel. More precisely, for general symmetric Markov processes it is known that the Nash inequality is equivalent to  on-diagonal upper heat kernel bounds, see \cite{CKS87}. This naturally includes the class of reversible random walks on graphs. On the other hand, in the study of the random conductance model (RCM), which is a prototype of a reversible random walk in random environment, some form of upper heat kernel bound is one of the crucial ingredients, which has been widely used in several subjects including quenched invariance principles \cite{ABDH13, BB07, BCKW21, SS04}, and quenched local limit theorems \cite{ADS16, AT21, ACS21, CKW24, BH09, CH08}. We refer the reader to \cite{Bi11, Ku14, An25} for detailed surveys on this topic.  On the other hand, some progress has also been made for random walks in the $\al$-stable regime, where for any $x, y \in \bbZ^{d}$, an edge between $x$ and $y$ is present with a probability proportional to $|x-y|^{-(d+\al)}$, $\al > 0$.  For $\al \in (0,2)$, the RCM on such graphs is known to be outside the Gaussian universality class, that is, the scaling limit is a symmetric $\al$-stable L\'{e}vy process instead of Brownian motion, see \cite{CS13, CKW21}, and the heat kernel exhibits non-Gaussian decay \cite{CS12, CKW20}. In this paper we consider a class of random walks between the nearest neighbour and the $\al$-stable regime, that is they admit long-range jumps but they still exhibit a Gaussian large-scale behaviour due to a suitable moment condition on the ergodic random conductances weighted by the squared size of the jump. For a particular sub-class of such models on $\bbZ^{d}$ a quenched invariance principle and local limit theorem have  been established in \cite{BCKW21, CKW24} for general ergodic conductances satisfying certain moment conditions.

On the integer lattice $\bbZ^{d}$, a Nash inequality holds for nearest neighbour random walks with jump rates that are uniformly bounded away from zero, see e.g.~\cite{Ba17} and references therein.  However, in the case where the weights are not bounded away from zero, no such Nash inequality is available.  An anchored version of the Nash inequality has been introduced by Mourrat and Otto in \cite{MO16} for nearest neighbour random walks on $\bbZ^{d}$ with jump rates unbounded from below. While the standard Nash inequality gives a control of the $L^{2}$-norm of a function $f$ in terms of $\| \nabla f \|_{2}$ and $\| f \|_{1}$, the anchored version provides control on the $L^{2}$-norm of $f$ in terms of $\| W \nabla f\|_{2}$, $\| f\|_1$ and $\| |x|^{m/2}f \|_{2}$, where $W$ is a weight function. Note that the translation invariance of the standard Nash inequality is broken due to the presence of the term $\| |x|^{m/2}f \|_{2}$. Nevertheless, it is still possible to deduce upper heat kernel bounds from it. Indeed, in \cite{MO16} on-diagonal upper bounds are obtained under a suitable integrability condition on the lower bound of the nearest neighbour weights, which are still required to be uniformly bounded from above.  Remarkably, their results can also be applied to nearest neighbour random walk on degenerate time-dependent conductances.

In this paper, we will extend some of the results in \cite{MO16} in various aspects. We prove an anchored Nash inequality for a non-local operator associated with a random walk performing nearest neighbour and long-range jumps, where we allow a general reference measure and weights that are unbounded from above and not bounded away from zero, only satisfying a certain integrablility condition. In addition, we work in a  setting of a more general class of underlying graphs including random graphs such as possibly correlated percolation clusters. Then we use this version of the anchored Nash inequality to derive quenched and annealed  on-diagonal heat kernel upper bounds for degenerate long-range RCMs.

\subsection{The model}
We consider an infinite connected locally finite graph $(\boldsymbol{V}, \boldsymbol{E})$ with vertex set $\boldsymbol{V}$ and non-oriented edge set $\boldsymbol{E}$. We write $x \sim y$ if $\{x, y\} \in \boldsymbol{E}$. Let $d\colon \boldsymbol{V} \times \boldsymbol{V} \to [0, \infty)$ be the graph distance on $\boldsymbol{V}$ defined by
\begin{align*}
  d(x,y)
  \;\ldef\;
  \inf\bigl\{
    n \geq 0 : \exists\, (x_i)_{0 \leq i \leq n}
    \text{ s.th. }
    x_{0} = x,\, x_{n}=y,\, x_{i} \sim x_{i-1},
    \forall\, 1 \leq i \leq n
  \bigr\}.
\end{align*}
Throughout the paper we fix an element $o \in \boldsymbol{V}$ and define $\eta\colon \boldsymbol{V} \to [0,\infty)$ by
\begin{align}\label{e1-4}
  \eta(y)
  \;\ldef\;
  \max\{d(o, y), 1\},
  \qquad x \in \boldsymbol{V}.
\end{align}
Let $B(x, R) \ldef \{y \in \boldsymbol{V} : d(x,y) \leq \lfloor R \rfloor\}$ for every $R \geq 0$ and we write $B(o, R)$ as $B(R)$ to simplify notation. For every $1 \leq p < \infty$ and any finite set $A \subset \boldsymbol{V}$, we define the normalized $\ell^{p}$-norm of any $f\colon A \to \R$ (w.r.t.\ the counting measure) as
\begin{align*} 
  \|f\|_{p,A}
  \;\ldef\;
  \Biggl(
    \frac{1}{|A|} \sum_{y \in A} \bigl|f(y)\bigr|^{p}
  \Biggr)^{\!1/p},
\end{align*}
where $|A|$ denotes the number of elements of subset $A\subset \boldsymbol{V}$. Suppose $\theta$ is a positive, locally finite measure on $\boldsymbol{V}$ and define $\theta(y) \ldef \theta(\{y\}) \in (0, \infty)$, $y \in \boldsymbol{V}$. Then for every $p \geq 1$ and any $f\colon \boldsymbol{V} \to \R$, we define the discrete $L^{p}$-norm of $f$ with respect to $\theta$ as
\begin{align*}
  \norm{f}{p}{\boldsymbol{V}\!,\, \theta}
  \;\ldef\;
  \Biggl(
    \sum_{x \in \boldsymbol{V}} \bigl|f(y)\bigr|^{p}\, \theta(y)
  \Biggr)^{\!1/p}.
\end{align*}
Suppose that $\{W(x,y)\}_{x, y \in \boldsymbol{V}}$ is a family of non-negative real numbers satisfying $W(x, y) > 0$ for $x \sim y$ and
\begin{align*}
  W(x,y) \;=\; W(y,x),
  \qquad
  \mu(x) \;\ldef\; \sum_{y \in \boldsymbol{V}} W(x,y) \;<\; \infty,
  \qquad \forall\, x, y \in \boldsymbol{V}.
\end{align*}
Then, for any $m \geq 0$ we introduce the measures $\mu_{m}$ and $\nu$ on $\boldsymbol{V}$ by
\begin{align}\label{def:mu_nu}
  \mu_{m}(x)
  \;\ldef\;
  \sum_{y\in \boldsymbol{V}} W(x,y)\, d(x, y)^{m}
  \qquad \text{and} \qquad
  \nu(x)
  \;\ldef\;
  \sum_{ y \sim x} W(x,y)^{-1}.
\end{align}
In particular, in the nearest neighbour case, that is when $W(x,y) \neq 0$ only if $x \sim y$,  we have $\mu = \mu_{m}$ for all $m \geq 0$.

Denote by $C_{c}(\boldsymbol{V})$ the collection of all functions on $\boldsymbol{V}$ having a finite support. Given conductances $\{W(x,y)\}_{x,y\in \boldsymbol{V}}$ as above, we define a quadratic form $\cE$ on $C_{c}(\boldsymbol{V})$ by
\begin{align}\label{e1-3}
  \cE(f)
  \;\ldef\;
  \frac{1}{2}\, \sum_{x,y \in \boldsymbol{V}} W(x,y)
  \bigl( f(y) - f(x) \bigr)^{2},
  \qquad f \in C_{c}(\boldsymbol{V}).
\end{align}
By a standard procedure, see e.g.~\cite[Example 1.2.4]{FOT11}, we can extend $(\cE, C_{c}(\boldsymbol{V}))$ to a regular Dirichlet form $(\cE, \cD(\cE))$ on $L^{2}(\boldsymbol{V}\!, \theta)$, and there exists a strong Markov process $((X_{t})_{t \geq 0}, \{\Prob_{\!x}\}_{x \in \boldsymbol{V}})$ associated with $(\cE, \cD(\cE))$. It is easy to verify that the infinitesimal generator $\cL$ of $X = (X_{t})_{t \geq 0}$ has the following form,
\begin{align} \label{eq:operator}
  \cL f(x)
  \;=\;
  \frac{1}{\theta(x)}\sum_{y \in \boldsymbol{V}} W(x,y)\, \big( f(y) - f(x) \big),\qquad f\in C_c(\boldsymbol{V}).
\end{align}
If the random walk $X$ is currently at $x$, it will next move to $y$ with probability $W(x,y)/\mu(x)$, after waiting an exponential time with mean $\theta(x)/\mu(x)$ at the vertex $x$. In particular, when $W(x,y) \neq 0$ only if $x \sim y$, $X$ is the nearest neighbour random walk associated with  conductances $\{W(x,y) : x, y \in \boldsymbol{V}\}$ and reference measure $\theta$.

Perhaps the most natural choice for the speed measure is $\theta \equiv \mu$, for which we obtain the constant speed random walk (CSRW) that spends i.i.d.\ $\operatorname{Exp}(1)$-distributed waiting times at all vertices it visits. Another well-studied process, the variable speed random walk (VSRW), is recovered by setting $\theta \equiv 1$, so called because as opposed to the CSRW, the waiting time at a vertex $x$ does indeed depend on the location; it is an $\operatorname{Exp}(\mu(x))$-distributed random variable. We write
\begin{align*}
  p(t,x,y) \;\ldef\; \frac{\Prob_{\!x}\bigl[X_{t} = y\bigr]}{\theta(y)}
\end{align*}
for the (minimal) heat kernel corresponding to $X$. For any subset $Q\subset \boldsymbol{V}$, let
\begin{align*}
  p_{Q}(t, x, y)
  \;\ldef\;
  \begin{cases}
    \dfrac{\Prob_{\!x}\bigl[X_{t} = y, t < \tau_{Q}\bigr]}{\theta(y)},
    & x, y \in Q,
    \\
    0,
    & \text{otherwise,}
  \end{cases}
\end{align*}
be the Dirichlet heat kernel on $Q$ corresponding to $X$, where $\tau_{Q} \ldef \inf\{t \geq 0 : X_{t} \notin Q\}$ is the first exit time from $Q$. We will also make the following assumption on the graph $(\boldsymbol{V}, \boldsymbol{E})$ and conductances $\{W(x,y) : x, y \in \boldsymbol{V}\}$.
\begin{assumption}\label{a2-2}
  There exist positive constants $d \geq 2$, $R_{0} \geq 1$ and $\alpha_{0} \in [0, 1/2)$ such that following (large scale) regularity conditions hold for $(\boldsymbol{V}, \boldsymbol{E})$.
  \begin{enumerate}[(i)]
  \item
    There exist $c_{\mathrm{reg}}, C_{\mathrm{reg}} \in (0,\infty)$ such that for all $R \geq R_{0}$ and $R^{\alpha_{0}} \leq r \leq R$,
    \begin{align}\label{a2-2-1}
      c_{\mathrm{reg}}r^{d}
      \;\leq\;
      |B(x,r)|
      \;\leq\;
      C_{\mathrm{reg}} r^{d},
      \qquad \forall\, x \in B(R).
    \end{align}

  \item
    For every $d' > d$, there exists $R_{1} \geq R_{0}$ such that for every $1 < \rho < d'$ and $\rho_{*} > \rho$ satisfying
    \begin{align}\label{a2-2-3}
      \frac{d'}{\rho_{*}}
      \;=\;
      \frac{d'}{\rho}-1,
    \end{align}
    there exist $C_{0}, C_{\mathrm{S}} \in (0,\infty)$ and $\delta_{\mathrm{S}} \in [1, \infty)$ such that for any
    $R \geq R_{1}$, $x \in B(R)$, $C_{0} R^{\alpha_{0}} \leq r \leq R$ and any $f\colon B(x, r) \to \R$,
    \begin{align}\label{a2-2-2}
      \norm{f - (f)_{B(x, r)}}{\rho_{*}}{B(x, r)}
      \;\leq\;
      C_{\mathrm{S}}\, r^{1 - d/d'}
      \Biggl(
        \sum_{\substack{y, z \in B(x, \delta_{\mathrm{S}} r) \\ y \sim z}}
        \mspace{-15mu}\bigl| f(y) - f(z) \bigr|^{\rho}
      \Biggr)^{\!1/\rho},
    \end{align}
    %
    %
    where $(f)_{B(x, R)} \ldef \frac{1}{|B(x, R)|}\sum_{y \in B(x,R)} f(y)$.
  \end{enumerate}
\end{assumption}
\begin{remark}\label{rem:graph:Zd}
  The Euclidean lattice, $(\bbZ^{d}, E_{d})$, satisfies the Assumption~\ref{a2-2} with $R_{0} = R_{1} = 1$ and $\alpha_{0} = 0$. Indeed, it is easy to see that the volume growth condition \eqref{a2-2-1} holds for every $x \in \bbZ^{d}$ and $r \geq 1$. On the other hand, recall that by \cite[Theorem~4.1]{Co96} or \cite[Theorem~2.6 and Theorem~2.8]{MO16}, for all $\rho', \rho'_{*} \in (0, \infty)$ satisfying
  \begin{align*}
    \frac{d}{\rho'_{*}} \;=\; \frac{d}{\rho'} - 1,
  \end{align*}
  there exists $C_{\mathrm{S}}\in (0, \infty)$ such that for every $x \in \bbZ^{d}$ and $r \geq 1$,
  \begin{align}\label{t1-2-2a}
    \norm{f - (f)_{B(x, r)}}{\rho_{*}'}{B(x, r)}
    \;\leq\;
    C_{\mathrm{S}}\,
    \Biggl(
      \sum_{\substack{y, z \in B(x,r) \\ y \sim z}}
      \mspace{-9mu}\bigl| f(y) - f(z) \bigr|^{\rho'}
    \Biggr)^{\!\!1/\rho'}.
  \end{align}
  %
  %
  Now, given positive constants $\rho_{*}$, $\rho$ satisfying \eqref{a2-2-3}, let $\rho_{1} \ldef d\rho / (d-\rho)$ so that $d/\rho_{1} = d/\rho - 1$. Then, by using H\"older's inequality and \eqref{t1-2-2a} (with $\rho'_{*}=\rho_1$ and $\rho' = \rho$) we obtain that for every $x \in \bbZ^{d}$ and $r \geq 1$,
  \begin{align*}
    \norm{f - (f)_{B(x, r)}}{\rho_{*}'}{B(x, r)}
    &\;\leq\;
    c_{1} r^{d(1/\rho_{*} - 1/\rho_{1})}\,
    \norm{f - (f)_{B(x, r)}}{\rho_{1}}{B(x, r)}
    \\
    &\;\leq\;
    c_{2} r^{1-d/d'}\,
    \Biggl(
      \sum_{\substack{y, z \in B(x,r) \\ y \sim z}}
      \mspace{-9mu}\bigl| f(y) - f(z) \bigr|^{\rho}
    \Biggr)^{\!\!1/\rho}.
  \end{align*}
  %
\end{remark}

\subsection{Main results}
As our first main result we establish an  anchored Nash inequality  with respect to a general reference measure $\theta$ under an integrability condition on $\theta$, $\theta^{-1}$ and $\nu$ in the present setting of a graph satisfying Assumption~\ref{a2-2}.
\begin{theorem}[Anchored Nash inequality]\label{thm:nash}
  Suppose that Assumption~\ref{a2-2} holds and suppose there exist $p, q \in (1, \infty]$ satisfying
  \begin{align}\label{a1-1-2}
    \frac{1}{p} + \frac{1}{q} \;<\; \frac{2}{d},
  \end{align}
  such that
  \begin{align}\label{a1-1-1}
    \sup_{R \geq 1} \Norm{\theta}{p, B(R)}
    \,+\,
    \sup_{R \geq 1} \Norm{\theta^{-1}}{1, B(R)}
    \,+\,
    \sup_{R \geq 1} \Norm{\nu}{q, B(R)}
    \;<\;
    \infty.
  \end{align}
  Then for every $m \geq 2$, there exists $C_{1} \in (0, \infty)$ such that
  \begin{align}\label{p2-1-1}
    \norm{f}{2}{\boldsymbol{V}, \theta}^2
    \;\leq\;
    C_{1} R_{1}^{m}\, \cM\, \cE(f)^{m/(m+2)}\,
    \norm{\eta^{m/2} f}{2}{\boldsymbol{V}, \theta}^{4/(m+2)},
  \end{align}
  where $\eta\colon \boldsymbol{V} \to [0, \infty)$ is defined by \eqref{e1-4}, $R_{1}$ is the same constant as in Assumption~\ref{a2-2} (associated with some $d'$) and
  \begin{align}\label{def:M}
    \cM
    \;\ldef\;
    \sup_{R \geq 1} \Bigl(
      \Norm{1 \vee \theta}{p, B(R)}
      \cdot \Norm{1 \vee \nu}{q, B(R)}
      \,+\,
      \Norm{1 \vee \theta^{-1}}{1, B(R)}
    \Bigr).
  \end{align}
\end{theorem}
In the context of random conductance models in ergodic environments the integrability condition \eqref{a1-1-1} can be easily verified  by using the ergodic theorem, provided a suitable moment condition is assumed. In fact, the condition on $p$ and $q$ in \eqref{a1-1-2} leads to the same moment condition under which various homogenization results have been established, see e.g.\ \cite{ADS15, ADS16,ADS16a,DNS18,AT21, ACS21, BCKW21}. Under the same assumptions we upgrade the Nash inequality in Theorem~\ref{thm:nash} to the following interpolated version.
\begin{theorem}[Interpolated anchored Nash inequality] \label{thm:nash_int}
  Suppose there exist $p, q\in (1, \infty]$ satisfying \eqref{a1-1-2} such that \eqref{a1-1-1} holds.
  \begin{enumerate}[(i)]
  \item
    Suppose that Assumption~\ref{a2-2} holds and $\inf_{y \in \boldsymbol{V}} \theta(y) > 0$. Then, for every $m \geq 2$, there exist $C_{2} \in (0, \infty)$, and $\alpha, \beta, \gamma \in (0, 1)$ with $\alpha + \beta + \gamma = 1$ such that for any $f\colon \boldsymbol{V} \to \bbR$,
    \begin{align}\label{p2-2-1}
      \norm{f}{2}{\boldsymbol{V}, \theta}^2
      \;\leq\;
      C_{\mathrm{AN}}\, \cM\, \cE(f)^{\alpha}\, \norm{f}{1}{\boldsymbol{V}, \theta}^{2\beta}\,
      \norm{\eta^{m/2} f}{2}{\boldsymbol{V}, \theta}^{\gamma},
    \end{align}
    where $C_{\mathrm{AN}} \ldef C_{2} R_{1}^{m}$ with $R_{1}$ being the same constant in Assumption~\ref{a2-2} (associated with some $d' > d$).

  \item
    Suppose that Assumption~\ref{a2-2} holds with $R_{0} = R_{1} = 1$ and $\alpha_{0} = 0$ (which means \eqref{a2-2-1} and \eqref{a2-2-2} hold for every $R \geq 1$, $x \in B(R)$ and $1 \leq r \leq R$). Then, for every $m \geq 2$, \eqref{p2-2-1} holds for every $f\colon \boldsymbol{V} \to \bbR$ (with $R_{1} = 1$).
  \end{enumerate}
\end{theorem}
In \cite{MO16} an (interpolated) anchored Nash inequality has been established for the counting measure on $\bbZ^{d}$, that is $\theta(x) = 1$ for all $x \in \boldsymbol{V} = \bbZ^{d}$, while in Theorems~\ref{thm:nash} and \ref{thm:nash_int} we obtain an anchored Nash inequality for a general class of reference measures $\theta$ and unbounded weights satisfying the integrability condition \eqref{a1-1-1}, on any graph $\boldsymbol{V}$ satisfying Assumption~\ref{a2-2}. This includes several degenerate models such as supercritical percolation clusters, see Subsection \ref{s1-4} below. This form of a Nash inequality allows us to deduce the following on-diagonal heat kernel upper bounds.
\begin{theorem}[Upper on-diagonal heat kernel bound]\label{t1-1}
  Let $d \geq 3$ and suppose that Assumption~\ref{a2-2} holds.
  \begin{enumerate}[(i)]
  \item
    Suppose there exist $m_{0} > d$ and $p \in \bigl(\max\{1,\frac{d}{m_{0}-2}\}, \infty\bigr]$, $q \in (1, \infty]$ satisfying \eqref{a1-1-2} such that
    \begin{align}\label{a1-4-1}
      \sup_{R \geq 1} \Norm{\theta}{p, B(R)}
      \,+\,
      \sup_{R \geq 1} \Norm{\mu_{m_{0}}}{p, B(R)}
      \,+\,
      \sup_{R \geq 1} \Norm{\nu}{q, B(R)}
      \;<\;
      \infty.
    \end{align}
    Then there exists $C_{3} \in (0, \infty)$ such that for every $t > 0$,
    \begin{align}\label{t1-1-1}
      p(t,o,o)
      \;\leq\;
      C_{3}\,  R_{0}^{(1/p + 1/q) m_{0}\gamma/(2\beta)}\, R_{1}^{m_{0}/\beta} 
      \cM_{0}^{(1 + m_{0}\gamma)/\beta}\,
      \biggl(
        1 + \frac{\mu_{m_{0}}(o)}{\theta(o)}
      \biggr)\, t^{-d/2},
    \end{align}
    where $R_{0}, R_{1} \geq 1$ are as in Assumption~\ref{a2-2} (for some $d' = d'(d, p, q) > d$),
    \begin{align}\label{t1-1-2}
      \cM_{0}
      \;\ldef\;
      \sup_{R \geq 1}\Bigl(
        \Bigl(
          \Norm{1 \vee \mu_{m_{0}}}{p, B(R)}
          \,+\, \Norm{1 \vee \theta}{p, B(R)}
        \Bigr) \cdot \Norm{1 \vee \nu}{q, B(R)}
      \Bigr),
    \end{align}
    and $\alpha, \beta, \gamma \in (0,1)$ are the same constants as in Lemma~\ref{l2-1} below associated with $m = m_{0}$.

  \item
    In the nearest neighbour case, that is when $W(x, y) \neq 0$ only if $x \sim y$, suppose there exist $p, q \in (1, \infty]$ satisfying \eqref{a1-1-2} such that the integrability condition \eqref{a1-4-1} holds (note that $\mu_{m_{0}}(x)=\mu(x)$ for any $m_{0}$ in the nearest neighbour case). Then, for every $m_{0} > d$,
    \begin{align}\label{t1-1-1a}
      p(t,o,o)
      \;\leq\;
      C_{3} \,
      R_{0}^{(1/p + 1/q) m_{0}\gamma/(2\beta)}\, R_{1}^{m_{0}/\beta}\,
      \cM_{0}^{(1 + m_{0}\gamma)/\beta}\,
      \biggl(
        1 + \frac{\mu(o)}{\theta(o)}
      \biggr)\, t^{-d/2}.
    \end{align}
  \end{enumerate}
\end{theorem}
Next we give a detailed discussion of our result in the following remark.
\begin{remark}
  (i) As mentioned above, in \cite{MO16} a uniform upper bound for weight $W$ is required in order to obtain an on-diagonal upper bound for nearest neighbour random walks via the corresponding (interpolated) anchored Nash inequality with respect to the counting measure. In Theorem~\ref{t1-1} this restriction is removed and the on-diagonal heat kernel bound holds, again for any reference measure $\theta$, only assuming the integrability condition \eqref{a1-4-1}. Moreover, compared with \cite{MO16}, our result also includes random walk with long-range jumps whose jump kernel is $L^{2}$-integrable, where the effect of the long-range jumps is controlled by the integrability of $\mu_{m_{0}}$ (defined in \eqref{def:mu_nu}).

  \smallskip 
  (ii) In order to derive the on-diagonal heat kernel upper bounds from the interpolated anchored Nash inequality,we follow the argument in \cite{MO16}. However, some additional challenges appear due to the presence of long-range jumps, and their interactions with general reference measures. In fact, we first introduce conditions \eqref{l2-2-1} and \eqref{l2-2-2} to control the effect of long-range jumps, which may fail for general reference measures without some uniform ellipticity assumptions. To resolve this issue, we apply Proposition~\ref{p4-1} not on the original walk $X$ but on a time-change of $X$ with speed measure $\pi_{m_{0}} = \max\{1, \mu_{m_{0}}, \theta\}$. Then, in transient dimensions, the resulting heat kernel bound for the time-changed process can be transferred to the original random walk $X$ by controlling the additive functional governing the time-change. \smallskip

  (iii) For the special case of a variable speed random walk on $\mathbb{Z}^{d}$ whose reference measure is the counting measure, an on-diagonal upper heat kernel estimate has been shown in \cite{CKW24} by using a maximal inequality with tail terms. However, it seems that the methods in \cite{CKW24} cannot be easily applied to random walks with long-range jumps having a general reference measure.
  \smallskip

  (iv) For nearest neighbour random walks on weighted, locally finite graphs, two-sided Gaussian heat kernel estimates have been proven by Delmotte \cite{De99} using discrete Moser iteration schemes. In \cite{Fo11}, off-diagonal Gaussian heat kernel upper bounds have been established for elliptic nearest neighbour random walk with arbitrary speed measure. Those bounds are obtained following an approach originating from \cite{Gr97}. Suitable on-diagonal estimates, which are the required input for this technique, can be extracted e.g.\ from the results in \cite{MO16}, even in the case without uniform ellipticity condition. Another useful technique  to prove off-diagonal Gaussian type upper bounds is known as Davies' method (see e.g.~\cite{Da89,Da93, CKS87}). It has also been successfully implemented in \cite{ADS16a,ADS19} to show Gaussian-type upper bounds for nearest neighbour random walk under similar integrability conditions as in the present work.  Note that the integrability conditions on $\omega(e)$ and $1/\omega(e)$ in this paper are necessary. Indeed, Gaussian bounds do not hold in the case of i.i.d.\ conductances with fat tails at zero due to a trapping phenomenon, see \cite{BBHK08}. But it seems that the aforementioned methods are not easily applicable to long-range random walks due to presence of long-range jumps with various lengths, see e.g.\ the analysis in \cite{BCKW21}.
  \smallskip
  
  (v) Unfortunately, the bounds in Theorem~\ref{t1-1} are not suitable to derive effective near-diagonal bounds of the form $p(t,o,x) \lesssim t^{-d/2}$, $x \in \boldsymbol{V}$, from them. This is mainly due to the fact that the constant in the right hand side of \eqref{t1-1-2} depends on the fixed point $o \in \boldsymbol{V}$. The lack of such near-diagonal bounds is also what hinders us to extend the bounds in Theorem~\ref{t1-1} to dimension $d = 2$. On the other hand, in situations where a classical Nash inequality of the form
  \begin{align}\label{e1-5}
    \norm{f}{2}{\boldsymbol{V}, \theta}^{2 + 4/d}
    \;\leq\;
    c\, \cE(f)\,
    \norm{f}{1}{\boldsymbol{V}, \theta}^{4/d},
  \end{align}
  is available, one can derive such bounds from it by a duality argument, see e.g.~\cite{CKS87}.
  
  One advantage of the bounds in Theorem~\ref{t1-1} is that, in the case where the weights are given by random ergodic conductances on $\bbZ^{d}$ satisfying a suitable explicit moment conditions, one can derive effective \emph{annealed} near diagonal bounds of the form $\mean\bigl[p^{\omega}(t,0, x)\bigr] \lesssim t^{-d/2}$ for all $x \in \bbZ^{d}$, see Theorem~\ref{t1-3} below. In fact, the approach based on Nash-type inequality leads to the constants $\cM$ and $\cM_{0}$. In the random setting the expectation of those can be easily controlled by the maximal ergodic theorem under an explicit moment condition. This is in contrast to e.g.\ the bounds obtained by Davies' perturbation techniques in \cite{ADS16a, ADS19}, where the analogue of those constants are much more intrinsic and a corresponding annealed results requires far more complicated moment conditions, cf.~\cite{AT21}.
  \smallskip
  
  (vi) For degenerate models, such as supercritical percolation clusters, one cannot expect the classical Nash inequality \eqref{e1-5} or Sobolev inequality to hold. Instead, in \cite{Ba04} regularity conditions based on volume doubling on mesoscopic scales and a local weak Poincar\'e inequality has been established in order to prove Gaussian heat kernel bounds for simple random walk on supercritical Bernoulli percolation clusters. This has been extended to more general degenerate models in \cite{BC16,Sa17}. As described in Subsection~\ref{s1-4} below, we will establish an interpolated anchored Nash inequality on a class of supercritical percolation clusters.  To our knowledge, this is the first global version of a functional inequality on those degenerate models in the literature. 
\end{remark}

\subsection{Application to random conductance models in ergodic environments}
As an application of our results we now present upper on-diagonal heat kernel estimates for random walks with long-range jumps among random ergodic conductances satisfying suitable moment conditions. More precisely, for $d \geq 2$ we consider the graph $(\boldsymbol{V}, \boldsymbol{E})$ with $\boldsymbol{V} = \bbZ^{d}$ and $\boldsymbol{E} = E_{d}$ being the collection all nearest neighbour bonds, that is, $E_{d} = \{\{x, y\} : x, y \in \bbZ^{d} \text{ with } |x-y|_{1} = 1\}$, where $|x|_{1} \ldef \sum_{i=1}^{d}|x_{i}|$ for every $x = (x_{1}, \cdots, x_{d}) \in \bbZ^{d}$. We still write $d(x,y)$ for the graph distance on $\bbZ^{d}$. Then, Assumption~\ref{a2-2} holds, cf.\ Remark~\ref{rem:graph:Zd} above.

Setting $\boldsymbol{\bar{E}} = \{\{x,y\} : x, y \in \bbZ^{d},\, x \ne y\}$, we place upon the graph positive weights $\omega = \{\omega(e) \in (0, \infty) : e \in \boldsymbol{\bar{E}} \}$, and for any $m \geq 0$ we define two measures on $\mathbb{Z}^{d}$,
\begin{align*}
  \mu^{\omega}_m(x)
  \;\ldef\;
  \sum_{y \in \bbZ^{d}} \omega(x,y)\, d(x, y)^{m},
  \qquad
  \nu^{\omega}(x)
  \;\ldef\;
  \sum_{y \sim x}\frac{1}{\omega(x,y)}.
\end{align*}
Let $(\Omega, \mathcal{F}) \ldef \bigl((0, \infty)^{\boldsymbol{\bar{E}}}, \mathcal{B}((0, \infty))^{\otimes \boldsymbol{\bar{E}}}\bigr)$ be the measurable space of all possible environments, i.e.\ all configurations of positive random weights. We denote by $\prob$ an arbitrary probability measure on $(\Omega, \mathcal{F})$ and $\mean$ the associated expectation. The measure space $(\Omega, \cF)$ is naturally equipped with a group of space shifts $\big\{\tau_{z} : z \in \bbZ^{d}\big\}$, which act on $\Omega$ as
\begin{align} \label{eq:def:space_shift}
  (\tau_z \omega)(x, y)
  \;\ldef\;
  \omega(\{x+z, y+z\}),
  \qquad \forall\, \{x, y\} \in \boldsymbol{\bar{E}}.
\end{align}
Then note that the measures $\mu^{\omega}_{m}$ and $\nu^{\omega}$ are stationary in the sense that $\mu_{m}^{\omega}(x+y) = \mu_{m}^{\tau_{y} \omega}(x)$ and $\nu^{\omega}(x + y) = \nu^{\tau_{y}\omega}(x)$ for all $x, y \in \bbZ^{d}$ and any $\omega \in \Omega$. Let $\theta^{\omega}\colon \bbZ^{d} \to (0, \infty)$ be a positive function which may depend on the environment $\omega \in \Omega$. For a fixed $\omega \in \Omega$, we denote by $X^{\omega} = (X_{t}^{\omega})_{t \geq 0}$ and $\{p^{\omega}(t, x, y)\}_{t > 0, x, y \in \bbZ^{d}}$ the Hunt process and associated heat kernel as constructed above for the choice $W(x,y) = W(y, x) = \omega(\{x, y\})$ and $\theta(x) \ldef \theta^{\omega}(x)$, $x, y \in \bbZ^{d}$. This means that the infinitesimal generator $\cL^{\omega}$ for $(X_{t}^{\omega})_{t \geq 0}$ acts on $f \in C_{c}(\bbZ^{d})$ as
\begin{align*}
  (\mathcal{L}^{\omega} f)(x)
  \:\ldef\;
  \frac{1}{\theta^{\omega}(x)}\, \sum_{y \sim x}\, \omega(\{x,y\})
  \bigl( f(y) - f(x)\bigr),
\end{align*}
and $X^{\omega}$ is reversible with respect to $\theta^{\omega}$. We call this process the \emph{random conductance model (RCM)} with \emph{speed measure} $\theta^{\omega}$. We denote $\Prob_{\!x}^{\omega}$ the law of this process started at $x \in \mathbb{Z}^{d}$ and $\Mean_{x}^{\omega}$ the corresponding expectation. There are two natural laws on the path space that are considered in the literature -- the quenched law $\Prob_{\!x}^{\omega}[\,\cdot\,]$ which concerns $\prob$-almost sure phenomena, and the annealed law $\mean\bigl[\Prob_{\!x}^{\omega}[\,\cdot\,]\bigr]$.
\begin{assumption}\label{ass:rcm}
  \begin{enumerate}[(i)]
  \item
    $\prob$ is stationary and ergodic with respect to spatial translations of $\mathbb{Z}^{d}$, that is, $\prob \circ \tau_x^{-1} = \prob$ for all $x \in \mathbb{Z}^{d}$ and $\prob[A] \in \{0,1\}$ for any $A \in \mathcal{F}$ such that $\tau_x(A) = A$ for all $x \in \mathbb{Z}^{d}$.

  \item
    $\theta^{\omega}$ is stationary, that is, $\theta^{\omega}(x + y) = \theta^{\tau_{y}\omega}(x)$ for all $x, y \in \bbZ^{d}$ and $\prob$-a.e.\ $\omega \in \Omega$.
  \end{enumerate}
\end{assumption}
As an immediate consequence of Theorem~\ref{t1-1} we get the following quenched heat kernel bound
\begin{theorem}[Quenched  heat kernel estimate] \label{thm:hke_quenched}
  Let $d \geq 3$ and suppose there exist $m_{0} > d$, $p \in \big(\max\{1, \frac{d}{m_{0}-2}\}, \infty\big]$ and $q \in (1, \infty]$ satisfying \eqref{a1-1-2} such that $\theta^{\omega}(0), \mu^{\omega}_{m_{0}}(0) \in L^{p}(\prob)$ and $\nu^{\omega}(0) \in L^{q}(\prob)$. Then, there exist random constant $\cM_{1}^{\omega}$ and non-random constant $C_{4} \in (0, \infty)$ such that for every $\omega \in \Omega$ and $t > 0$,
  \begin{align}\label{t1-3-1a}
    p^{\omega}(t,0,0)
    \;\leq\;
    C_{4}\, \cM_{1}^{\omega}\, t^{-d/2},
  \end{align}
  where $\cM_{1}^{\omega}$ is $\prob$-a.s.\ finite and explicity given by
  \begin{align}\label{t1-3-1aa}
    \cM_{1}^{\omega}
    \;\ldef\;
    \bigl(\cM_{0}^{\omega}\bigr)^{(1+m_{0}\gamma)/\beta}
    \bigg(1 + \frac{\mu_{m_{0}}^{\omega}(0)}{\theta^{\omega}(0)}\bigg)
  \end{align}
  with
  \begin{align}\label{t1-3-2}
    \cM_{0}^{\omega}
    \;\ldef\;
    \sup_{R \geq 1}\Bigl(
      \Bigl(
        \Norm{1 \vee \mu_{m_{0}}^{\omega}}{p, B(R)}
        + \Norm{1 \vee \theta^{\omega}}{p, B(R)}
      \Bigr) \cdot \Norm{1 \vee \nu^{\omega}}{q, B(R)}
    \Bigr),
  \end{align}
  and $\alpha, \beta, \gamma \in (0,1)$ are the same constants as in Lemma~\ref{l2-1} below associated with $m = m_{0}$.
\end{theorem}
\begin{proof}
  Since we have assumed that $\theta^{\omega}(0), \mu^{\omega}_{m_{0}}(0) \in L^{p}(\prob)$ and $\nu^{\omega}(0) \in L^{q}(\prob)$, an application of the maximal ergodic theorem gives that $\cM_{0}^{\omega}$ defined by \eqref{t1-3-2} is $\prob$-a.s.\ finite.

  As explained in Remark~\ref{rem:graph:Zd} above, Assumption~\ref{a2-2} holds with $R_{0} = R_{1} = 1$ and $\alpha_{0} = 0$. Hence by (the proof of) Theorem \ref{thm:nash_int} and Theorem~\ref{t1-1}, we know that the positive constants $C_{2}$, $C_{3}$ in \eqref{p2-2-1} and \eqref{t1-1-1} are non-random, so the claim is immediate from Theorem~\ref{t1-1} (with associated constant $C_{4}$ being non-random).
\end{proof}
For the VSRW, that this $\theta^{\omega}=1$, under stronger conditions on the environment, near-diagonal upper bounds on the parabolic scale have been show in \cite{CKW24}. More precisely, assuming that, for some $m_{0} \geq d+2$, $\mu^{\omega}_{2}(0), \mu_{m_{0}}^{\omega}(0) \in L^{p}(\prob)$ and $\nu^{\omega}(0) \in L^{q}(\prob)$ for $p, q \in (1, \infty]$ satisfying $\left(1-\frac{1}{d}\right)\cdot \frac{1}{p} + \frac{1}{q} \leq \frac{1}{d}$, we have $p^{\omega}(t, 0, x) \lesssim t^{-d/2}$ for all $t$ larger than some (non-explicit) random constant and all $x$ in ball centred at zero with radius of order $\sqrt{t}$, see \cite[Proposition~3.7]{CKW24}. Besides the fact that the on-diagonal bound in Theorem~\ref{thm:hke_quenched} holds for general speed measures under weaker moment assumptions, another advantage of Theorem~\ref{thm:hke_quenched} in comparison with the result of \cite{CKW24} is that, as already mentioned above, the random constant $\cM_{1}^{\omega}$ is explicit, which allows to derive the following annealed heat kernel upper bound.
\begin{theorem}[Annealed heat kernel bound]\label{t1-3}
  Let $m_{0} > d$, $p \in \big(\max\{1,\frac{d}{m_{0}-2}\}, \infty\big]$ and $q \in (1,\infty]$ be satisfying \eqref{a1-1-2}.
  \begin{enumerate}[(i)]
  \item
    For $d \geq 3$ and suppose there exist $p_{1}, p_{2} \in (1, \infty]$ and $q_{1}, q_{2} \in (1, \infty]$ satisfying
    \begin{align}\label{t1-3-2a}
      \frac{1}{p_1} + \frac{1}{p_2} \;=\; 1,
      \qquad
      \frac{1}{q_1} + \frac{1}{q_2} \;=\; 1,
    \end{align}
    such that
    \begin{align}\label{t1-3-3}
      \theta^{\omega}(0), \mu^{\omega}_{m_{0}}(0) \in L^{\tilde p}(\prob),
      \qquad
      \nu^{\omega}(0) \in L^{\tilde q}(\prob),
      \qquad
      \frac{\mu^{\omega}_{m_{0}}(0)}{\theta^{\omega}(0)}\in L^{q_2}(\prob),
    \end{align}
    where
    \begin{align} \label{eq:defTildep0}
      \tilde p\ldef\max\Big\{p,p_1q_1\Big(1+\frac{1+m_{0}\gamma}{\beta}\Big)\Big\},\quad \tilde q \ldef \max\Big\{q,p_2q_1\Big(1+\frac{1+m_{0}\gamma}{\beta}\Big)\Big\}.
    \end{align}
    Then there exists a constant $C_5\in (0,\infty)$ such that
    \begin{align}\label{t1-3-4}
      \mean\bigl[p^{\omega}(t, x, y) \bigr]
      \;\leq\;
      C_{5}\, t^{-d/2},
      \qquad  \forall\, x, y \in \bbZ^{d}, t > 0.
    \end{align}
    
  \item
    Let $d=2$ and  suppose there exist $p_1,p_2\in (1, \infty]$ and $q_1,q_2\in (1, \infty]$ satisfying \eqref{t1-3-2a}
    such that
    \begin{align}
      \label{t1-3-3a}
      \theta^{\omega}(0), \mu^{\omega}_{m_{0}}(0) \in L^{2\tilde{p}}(\prob),
      \quad
      \nu^{\omega}(0) \in L^{2\tilde{q}}(\prob),
      \quad
      \frac{\mu^{\omega}_{m_{0}}(0)}{\theta^{\omega}(0)} \in L^{2q_2}(\prob),
    \end{align}
    and
    \begin{align} \label{t1-3-3a1}
      \frac{
        \bigl(
          1 \vee \theta^{\omega}(0) \vee \mu^{\omega}_{m_{0}}(0)
        \bigr)^2}{\theta^{\omega}(0)} \in L^{1}(\prob)
    \end{align}
    with $\tilde{p}$ and $\tilde{q}$ as defined in \eqref{eq:defTildep0}. Then there exists $C_{6} \in (0, \infty)$ such that
    \begin{align}\label{t1-3-4a}
      \mean\bigl[p^{\omega}(t, x, y)\bigr]
      \;\leq\;
      C_{6}\,
      \Bigl(
        1 + \log t\,
        \indicator_{\{t \geq 2\}}
      \Bigr)\, t^{-1},
      \qquad \forall\, x,y \in \bbZ^{d}, t > 0.
    \end{align}
  \end{enumerate}
\end{theorem}
\begin{remark}
  Similarly as in Theorem~\ref{t1-1}-(ii), in the nearest neighbour case the results in Theorems~\ref{thm:hke_quenched} and \ref{t1-3} hold for any $p,q \in (1,\infty]$ satisfying $1/p + 1/q < 2/d$ and any $m_{0} > d$.
\end{remark}
Annealed heat kernel estimates and annealed homogenization results such as invariance principles and local limit theorems are crucial for applications to certain models in statistical mechanics. One prime example is the Ginzburg-Landau $\nabla \phi$ interface model, linked to random conductance models via the Helffer-Sj\"ostrand representation, see e.g.\ \cite[Section~5]{DD05} for details. Moreover, various techniques in quantitative stochastic homogenization theory (cf. e.g. \cite{AKM19, GNO15}) rely on annealed heat kernel estimates. On one hand, a quenched invariance principle does imply an annealed invariance principle in general. On the other hand, the same does not apply to local limit theorems or estimates for the annealed heat kernel. In fact, establishing such annealed results often requires different techniques and significantly stronger assumptions (cf.\ e.g.\ \cite{DD05, MO15, An14, AT21}), for instance in form of moment conditions, and Theorem~\ref{t1-3} may be regarded as another item in this collection. Recently, in \cite{DKS23}, for nearest-neighbour random walks under time-dependent conductances that are bounded from below but unbounded from above, only required to have a finite first moment, sharp annealed on-diagonal estimates for the first and second discrete derivative of the heat kernel are obtained, which then are used to prove a local limit theorem for the annealed heat kernel and its discrete first derivative as well as optimal decay rates for the annealed Green's function and its derivatives.

\subsection{Applications on supercritical percolation clusters}\label{s1-4}
We consider bond percolation on $(\bbZ^{d}, E_{d})$ with $d \geq 2$. Let $\Omega \ldef \{0, 1\}^{E_{d}}$ endowed with the canonical coordinate map $e \mapsto \omega(e)$ for every $e \in E_{d}$, and
let $\cF$ denote the collection of all the subsets of $\Omega$. In particular, if $\omega(e) = 1$ for any edge $e \in E_{d}$, then we say that $e$ is \emph{open} under the realization $\omega \in \Omega$, and denote by $\mathcal{O}(\omega)$ the set of open edges. Moreover, let $\mathcal{C}_{\infty}(\omega)$ be the subset of vertices of $\mathbb{Z}^{d}$ that belongs to infinite components, each of which is connected by paths along open edges. Let $\prob$ be probability measure on $(\Omega,\cF)$. Throughout this subsection we will make the following assumption.
\begin{assumption}\label{a1-4}
  (i) The law $\prob$ is stationary and ergodic with respect to space shifts of $\mathbb{Z}^{d}$, that is, $\prob \circ \tau_{x}^{-1} = \prob$ for all $x \in \mathbb{Z}^{d}$ and $\prob[A] \in \{0, 1\}$ for any $A \in \mathcal{F}$ such that $\tau_{x}^{-1}(A) = A$ for all $x \in \mathbb{Z}^{d}$.

  (ii) For $\prob$-a.e.\ $\omega$, the set $\mathcal{C}_{\infty}(\omega)$ is non-empty and there exists a unique infinite connected component, which is also called infinite open cluster, and $\prob\bigl[0 \in \mathcal{C}_{\infty}\bigr] > 0$.
\end{assumption}
We write $\prob_{0}[ \,\cdot\, ] \ldef \prob\bigl[ \,\cdot \,|\, 0 \in \mathcal{C}_{\infty}\bigr]$ to denote the conditional distribution given the event $\{0 \in \mathcal{C}_{\infty}\}$, and $\mean_{0}$ for the expectation with respect to $\prob_{0}$. It is well known that Assumption~\ref{a1-4} holds in the case of i.i.d.\ Bernoulli bond percolation on $\bbZ^{d}$ if $\prob\bigl[\omega(e) > 0\bigr] > p_{c}$, where $p_{c} \equiv p_{c}(d)$ denotes the critical probability for bond percolation on $\mathbb{Z}^{d}$, but it also holds for a class of correlated percolation models as studied e.g.\ in \cite{Sa17}, including level sets of the discrete Gaussian free fields and random interlacements. We refer to \cite{Sa17} for a detailed introduction to such percolation models with long-range correlations.

Let $E(\mathcal{C}_{\infty}(\omega)) \ldef \{\{x, y\} \in \mathcal{O}(\omega) : x, y \in \mathcal{C}_{\infty}(\omega)\}$ be the edge set of the infinite open cluster. We denote by $d^{\omega}$ the graph distance on the pair $\bigl(\mathcal{C}_{\infty}(\omega), E(\mathcal{C}_{\infty}(\omega))\bigr)$, namely, for any $x, y \in \mathcal{C}_{\infty}(\omega)$, $d^{\omega}(x, y)$ is the minimal length of a path joining $x$ and $y$ that consists only of open edges. For $R \geq 1$ and $x \in \cC_{\infty}(\omega)$, let
\begin{align*}
  B^{\omega}(x, R)
  \;\ldef\;
  \bigl\{y \in \cC_{\infty}(\omega) : d^{\omega}(x,y) \leq \lfloor R \rfloor \bigr\},
  \qquad
  Q^{\omega}(x,R)
  \;\ldef\;
  \cC_{\infty}(\omega) \cap Q(x, R),
\end{align*}
where $Q(x, R) \ldef \{z \in \bbZ^{d}: |x-z|_{\infty} \leq \lfloor R \rfloor\}$ with $|x|_{\infty} \ldef \max_{1 \leq i \leq d} |x_i|$, $x \in \bbZ^{d}$, and we write $x \sim y$ for $x, y \in \cC_\infty(\omega)$ if $d^{\omega}(x,y) = 1$. Finally, for any $A \subset B \subset \mathbb{Z}^{d}$ we define the \emph{relative} boundary of $A$ with respect to $B$ by
\begin{align*}
  \partial_{\!B}^{\omega} A
  \;\ldef\;
  \big\{%
    \{x,y\} \in \mathcal{O}(\omega) :
    x \in A \,\text{ and }\, y \in B \setminus A
  \big\},
\end{align*}
and we simply write $\partial^{\omega} A$ if $B \equiv \mathcal{C}_{\infty}(\omega)$. We will assume that the following volume regularity and isoperimetric inequality on large scales hold.
\begin{assumption}\label{a1-5}
  There exists $\alpha_{0} \in (0, 1/2)$ such that, for $\prob_{0}$-a.e.\ $\omega \in \Omega$ and any $d' > d$, there exist a random constant $R_{2}(\omega) > 0$ and non-random constants $C_{7}, C_{8}, C_{9} \in (0,\infty)$ such that for every $R \geq R_{2}(\omega)$ and $C_{7} R^{\alpha_{0}} \leq r \leq R$, the following
  properties hold.
  \begin{enumerate}[(i)]
  \item
    For any $x, y \in B^{\omega}(0, R)$,
    \begin{align}\label{l5-1-3}
      C_{8}\, |x - y|_{\infty}
      \;\leq\;
      d^{\omega}(x,y)
      \;\leq\;
      C_{9}\, \max\bigl\{|x-y|_{\infty}, R^{\alpha_{0}} \bigr\}.
    \end{align}

  \item
    For any $x \in B^{\omega}(0, R)$,
    \begin{align}\label{l5-1-1}
      C_{8}\, r^{d} \;\leq\; |B^{\omega}(x,r)| \;\leq\; C_{9}\, r^{d}
      \qquad \text{and} \qquad
      C_{8}\, r^{d} \;\leq\; |Q^{\omega}(x,r)| \;\leq\; C_{9}\, r^{d}.
    \end{align}

  \item
    For any $x \in B^{\omega}(0, R)$,
    \begin{align}\label{l5-1-2}
      \inf\Biggl\{
        \frac{|\partial^{\omega}_{Q^{\omega}(x,r)} A|}{|A|^{1-\frac{1}{d'}}}
        : A \subset Q^{\omega}(x,r),\
        |A| \leq \frac{|Q^{\omega}(x,r)|}{2}
      \Biggr\}
      \;\geq\;
      C_{8} r^{\frac{d}{d'}-1}.
    \end{align}
  \end{enumerate}
\end{assumption}
\begin{remark}
  (i) For the Bernoulli bond percolation, the large scale comparability \eqref{l5-1-3} between $\ell^\infty$ distance $|x-y|_{\infty}$ and the graph distance $d^{\omega}(x,y)$ has been established in \cite[Proposition~2.17(d)]{Ba04}. The first inequality in \eqref{l5-1-1} can be obtained directly from \cite[Theorem~2.18]{Ba04} while the second inequality in \eqref{l5-1-1} is a combination of \cite[Theorem~2.18]{Ba04} and
  \eqref{l5-1-3}. The estimate \eqref{l5-1-2} with $x=0$ and $r=R$
  has been shown by \cite[Section~3 (16)]{MR04}. Following the same renormalization arguments as in the proof of \cite[Proposition~2.17, Theorem~2.18]{Ba04} we can prove \eqref{l5-1-2} for every $x \in B^{\omega}(0, R)$ and $R^{\alpha_{0}} \leq r \leq R$.

  (ii) For the aforementened percolation models with long-range correlations, Assumption~\ref{a1-5} can also be verified, see \cite[Lemma~3.15, Corollaries~3.16 and 3.17]{Sa17}.
\end{remark}
For $\prob_{0}$-a.e.\ $\omega \in \Omega$, consider a symmetric weight function $W^{\omega}\colon \cC_{\infty}(\omega) \times \cC_{\infty}(\omega) \to (0, \infty)$. Then, for every $m \geq 1$, we define the measure $\mu_{m}^{\omega}, \nu^{\omega}\colon \cC_{\infty}(\omega) \to (0,\infty)$ as
\begin{align*}
  \mu^{\omega}_{m}(x)
  \;\ldef\;
  \sum_{y \in \cC_{\infty}(\omega)} \mspace{-6mu}W^{\omega}(x,y)\, d^{\omega}(x,y)^{m},
  \qquad
  \nu^{\omega}(x)
  \;\ldef\;
  \sum_{\substack{y \in \cC_{\infty}(\omega) \\ y \sim x}} \frac{1}{W^{\omega}(x,y)}.
\end{align*}
Hence taking $\boldsymbol{V} = \cC_{\infty}(\omega)$ and $W(x,y) = W^{\omega}(x,y)$ and fixing a positive reference measure $\theta^{\omega}\colon \cC_{\infty}(\omega)\to (0,\infty)$ on $\cC_{\infty}(\omega)$, there exists an associated Hunt process $(X_{t}^{\omega})_{t \geq 0}$ with heat kernel $\{p^{\omega}(t, y, z)\}_{t>0;y,z\in \cC_{\infty}(\omega)}$. In particular, the infinitesimal generator $\mathcal{L}^{\omega}$ of $(X_t^{\omega})_{t \geq 0}$ acts on functions $f \in C_{c}(\cC_{\infty}(\omega))$ as
\begin{align*}
  (\mathcal{L}^{\omega} f)(x)
  \;\ldef\;
  \frac{1}{\theta^{\omega}(x)} \sum_{y \in \cC_{\infty}(\omega)}
  \mspace{-6mu}W^{\omega}(x,y)\, \bigl(f(y) - f(x)\bigr).
\end{align*}
Then, the above results will allow us to deduce the following (quenched) anchored Nash inequality on $\cC_{\infty}(\omega)$ and an on-diagonal heat kernel estimate.
\begin{theorem}\label{t1-4}
  Suppose that Assumptions~\ref{a1-4} and \ref{a1-5} hold and for $\prob_{0}$-a.e.\ $\omega$,
  \begin{align}\label{t1-4-3}
    \inf_{y \in \cC_{\infty}(\omega)} \theta^{\omega}(y) > 0.
  \end{align}
  Further, assume that there exist $p, q \in (1, \infty]$ satisfying \eqref{a1-1-2} such that for $\prob_{0}$-a.e.\ $\omega$,
  \begin{align}\label{t1-4-2}
    \sup_{R \geq 1} \Bigl(
      \Norm{1 \vee \theta^{\omega}}{p, B^{\omega}(0, R)}
      + \Norm{1 \vee \nu^{\omega}}{q, B^{\omega}(0, R)}
    \Bigr)
    \;<\;
    \infty.
  \end{align}
  Then, for every $m \geq 2$ and $\prob_{0}$-a.e.\ $\omega$, there exist $C_{10}(\omega) \in (0, \infty)$ and $\alpha, \beta, \gamma \in (0, 1)$ (independent of $\omega$) with $\alpha + \beta + \gamma = 1$ such that for every $f\colon \cC_{\infty}(\omega) \to \bbR$,
  \begin{align}\label{t1-4-1}
    \norm{f}{2}{\cC_{\infty}(\omega), \theta^{\omega}}^2
    \;\leq\;
    C_{10}(\omega)\, \cE^{\omega}(f)^{\alpha}\,
    \norm{f}{1}{\cC_{\infty}(\omega), \theta^{\omega}}^{2\beta}\,
    \norm{(\eta^{\omega})^{m/2} f}{2}{\cC_{\infty}(\omega), \theta^{\omega}}^{\gamma},
  \end{align}
  where $\eta^{\omega}(y) \ldef \max\{1, d^{\omega}(0, y)\}$ and
  \begin{align*}
    \cE^{\omega}(f)
    \;\ldef\;
    \frac{1}{2}\,
    \sum_{\substack{x,y \in \cC_{\infty}(\omega) \\ y \sim x}}
    \mspace{-12mu}W^{\omega}(x,y)\, \bigl(f(y) - f(x)\bigr)^2.
  \end{align*}
\end{theorem}
\begin{theorem}\label{t1-5}
  Let $d \geq 3$. Suppose that Assumptions~\ref{a1-4} and \ref{a1-5} hold, and there exist $m_{0} > d$ and $p \in \bigl(\max\{1, d/(m_{0}-2)\}, \infty\bigr]$, $q \in (1, \infty]$ satisfying \eqref{a1-1-2} such that, for $\prob_{0}$-a.e.\ $\omega$, \eqref{t1-4-3}, \eqref{t1-4-2} and
  \begin{align}\label{t1-5-1}
    \sup_{R \geq 1} \Norm{1 \vee \mu_{m_{0}}^{\omega}}{p, B^{\omega}(0, R)}
    \;<\;
    \infty
  \end{align}
  hold. Then, for $\prob_{0}$-a.e.\ $\omega$, there exists a random constant $C_{11}(\omega) \in (0, \infty)$ such that for all $t > 0$,
  \begin{align}\label{t1-5-2}
    p^{\omega}(t, 0, 0) \;\leq\; C_{11}(\omega)\, t^{-d/2}.
  \end{align}
\end{theorem}
\begin{remark}
  If the weights $W^{\omega}(x,y)$ and the speed measure $\theta^{\omega}$ are again assumed to be stationary and ergodic random variables, then conditions \eqref{t1-4-2} and \eqref{t1-5-1} can be replaced by the  simpler moment conditions $\mean[\theta^{\omega}(0)^p] < \infty$, $\mean[\nu^{\omega}(0)^{q}] < \infty$, and $\mean[\mu_{m_{0}}^{\omega}(0)^p] < \infty$, respectively, by invoking the ergodic theorem. Indeed, note that for sufficiently large $R$, using the volume regularity and the fact that $B^{\omega}(0, R) \subseteq B(0, R)$ we have $\prob_{0}$-almost surely,
  \begin{align*}
    \frac{1}{|B^{\omega}(0,R)|}
    \sum_{y \in B^{\omega}(0,R)} \theta^{\omega}(y)^{p} 
    \;\leq\;
    c R^{-d} \sum_{y \in B(0,R)} \theta^{\omega}(y)^{p} 
    \;\underset{R\to \infty}{\longrightarrow}\;
    c\, \mean[\theta^{\omega}(0)^p],
  \end{align*}
  and similarly for the other terms involving $\nu^{\omega}$ and $\mu_{m_{0}}^{\omega}$.
\end{remark}
The rest of the paper is organised as follows. In Section~\ref{sec:Nash} we prove the anchored Nash inequalities in Theorem~\ref{thm:nash} and Theorem~\ref{thm:nash_int}, and  in Section~\ref{sec:hke} we deduce the heat kernel bound in Theorem~\ref{t1-3-1a} from them. The annealed heat kernel bound for the random conductance model is shown in Section~\ref{sec:rcm}. The appendix contains a technical lemma. Throughout the paper we write $c$ to denote a positive constant which may change on each appearance. Constants denoted $C_{i}$ will remain the same.

\section{Anchored Nash inequality} \label{sec:Nash}
\subsection{Notation and preliminaries}
For $\phi\colon \boldsymbol{V} \to [0, \infty)$, $p \in [1, \infty)$ and any non-empty, finite subset $B \subset \boldsymbol{V}$, we define space-averaged weighted $\ell^{p}$-norms
{\color{blue} for}
functions $f\!: B \to \bbR$ by
\begin{align*}
  \Norm{f}{p, B, \phi}
  \;\ldef\;
  \bigg(
    \frac{1}{|B|}\; \sum_{x \in B}\, |f(x)|^p\, \phi(x)
  \bigg)^{\!\!1/p}
  \qquad \text{and} \qquad
  \Norm{f}{\infty, B} \;\ldef\; \max_{x \in B} |f(x)|.
\end{align*}
If $\phi \equiv 1$, we still simply write $\|f\|_{p, B} \ldef \|f\|_{p, B, \phi}$. Moreover, we define  $(f)_{B} \equiv (f)_{B,1}$ and $(f)_{B, \phi} \ldef \sum_{x \in B} \phi(x) f(x) / \sum_{x \in B} \phi(x)$ for any non-negative weight $\phi$ on $\boldsymbol{V}$.
Then by H\"older's inequality, it is easy to see that for every $1\le p_1\le p_2<\infty$ and any finite subset $A\subset \boldsymbol{V}$,
\begin{align}\label{e1-2a}
\|f\|_{p_1,A}\le \|f\|_{p_2,A}.
\end{align}

\subsection{Anchored Nash inequality}
In this subsection we will prove the anchored Nash inequality in Theorem~\ref{thm:nash}.
\begin{lemma}\label{l1-1}
  Suppose that Assumption~\ref{a2-2} holds. Then, for every $d' > d$ there exists a constant $C_{12}>0$ such that for every $R \geq R_1$ (here $R_{1}$ is the same constant in Assumption~\ref{a2-2} associated with $d'$) and $f, g\colon B(R) \to \R$ with $(g)_{B(R)} = 1$,
  \begin{align}\label{l1-1-2}
    &\norm{f - (fg)_{B(R)}}{\rho_{*}}{B(R)}
    \nonumber\\
    &\mspace{36mu}\leq\;
    C_{12}\, R^{1-d/d'}\,
    \Bigl(
      1 + \Norm{g}{\rho_{*}/(\rho_{*}-1), B(R)}
    \Bigr)\,
    \Biggl(
      \sum_{\substack{y, z \in B(\delta_{\mathrm{S}} R) \\ y \sim z}}
      \mspace{-12mu}\bigl|f(y) - f(z)\bigr|^{\rho}
    \Biggr)^{\!1/\rho}.
  \end{align}
  %
  %
  Here $\rho, \rho_{*} \in (1, \infty)$ and $\delta_{\mathrm{S}} \in [1, \infty)$ are the same constants as in Assumption~\ref{a2-2}-(ii) satisfying \eqref{a2-2-3} and \eqref{a2-2-2}.
\end{lemma}
\begin{proof}
  Fix some $R \geq R_{0}$. Since $(g)_{B(R)} = 1$, an application of Minkowski's and H\"older's inequality yields,
  \begin{align*}
    &\norm{f - (fg)_{B(R)}}{\rho_{*}}{B(R)}
    \\[.5ex]
    &\mspace{36mu}\leq\;
    \norm{f - (f)_{B(R)}}{\rho_{*}}{B(R)}
    \,+\, |B(R)|^{1/\rho_{*}}\,
    \bigl|(fg)_{B(R)} - (f)_{B(R)}\bigr|
    \\[.5ex]
    &\mspace{36mu}\leq\;
    \Bigl(1 + \Norm{g}{\rho_{*}/(\rho_{*}-1), B(R)}\Bigr)\,
    \norm{f - (f)_{B(R)}}{\rho_{*}}{B(R)}
  \end{align*}
  Thus, by combining this with \eqref{a2-2-2} (with $x=o$ and $r=R$) the assertion \eqref{l1-1-2} follows.
\end{proof}
\begin{proof}[Proof of Theorem~\ref{thm:nash}]
  Due to \eqref{a1-1-2} there exists $d' > d$ such that
  \begin{align}\label{p2-1-6}
    \frac{1}{p} + \frac{1}{q} \;<\; \frac{2}{d'}.
  \end{align}
  Let $R_{1} \geq 2$ be the constant in Assumption~\ref{a2-2} associated with this $d'$. Set $\rho_{*} = 2p/(p-1)$ and $\rho = \rho_{*} d'/(\rho_{*}+d') = 2p d'/(pd' + 2p - d')$. Fix $R \geq R_{1}$ and let $A(R) \ldef B(R) \setminus B(\kappa R)$ and
  \begin{align*}
    g(x)
    \;\ldef\;
    \begin{cases}
      \dfrac{|B(R)|}{|A(R)|}, & x \in A(R),
      \\
      0, & \text{otherwise},
    \end{cases}
  \end{align*}
  where $\kappa \in (0,1)$ is a (small) constant to be determined later. In particular, $(g)_{B(R)} = 1$. Taking $\kappa \ldef 2^{-d} (c_{{\rm reg}}/C_{\mathrm{reg}})^{1/d}$ with $c_{\mathrm{reg}}$ and $C_{\mathrm{reg}}$ as in \eqref{a2-2-1}, the volume regularity \eqref{a2-2-1} implies that, for every $R \geq R_{1}$ (after choosing a larger $R_{1}$ if necessary),
  \begin{align}\label{p2-1-3}
    |A(R)|
    \;=\;
    |B(R)| - |B(\kappa R)|
    \;\geq\;
    c_{\mathrm{reg}} R^{d} - C_{\mathrm{reg}}(\kappa R)^{d}
    \;\geq\;
    \frac{c_{\mathrm{reg}} R^{d}}{2}.
  \end{align}
  Hence, $\Norm{g}{\rho_{*}/(\rho_{*}-1), B(R)} \leq c$ for all $R \geq R_{1}$. Moreover, in view of Lemma~\ref{l1-1} we obtain that
  \begin{align}\label{p2-1-2}
    &\norm{f - (fg)_{B(R)}}{\rho_{*}}{B(R)}
    \nonumber\\
    &\mspace{36mu}\leq\;
    C_{12}\, R^{1-d/d'}\, \Bigl(
      1 + \Norm{g}{\rho_{*}/(\rho_{*}-1), B(R)}
    \Bigr)\,
    \Biggl(
      \sum_{\substack{y, z \in B(\delta_{\mathrm{S}} R) \\ y \sim z}}
      \mspace{-12mu}\bigl|f(y) - f(z)\bigr|^{\rho}
    \Biggr)^{\!1/\rho}
    \nonumber\\
    &\mspace{36mu}\leq\;
    c R^{1-d/d'}\,
    \bigl|B(\delta_{\mathrm{S}} R)\bigr|^{(2-\rho)/(2\rho)}\,
    \cE(f)^{1/2}\,
    \Norm{\nu}{\rho/(2-\rho), B(\delta_{\mathrm{{S}}} R)}^{1/2},
  \end{align}
  where we used H\"older's inequality, the definitions of $\cE(f)$  and $\nu$ in \eqref{e1-3} and \eqref{def:mu_nu}, respectively. Since $\rho/(2-\rho) < q$, which is ensured by \eqref{p2-1-6}, an application of Jensen's inequality $\Norm{\nu}{\rho/(2-\rho), B(\delta_{\mathrm{S}} R)} \leq \Norm{\nu}{q, B(\delta_{\mathrm{S}} R)}$. On the other hand, by the definition of $g$ and \eqref{p2-1-3},
  \begin{align}\label{eq:est_av}
    \bigl|(fg)_{B(R)}\bigr|
    \;\leq\;
    \Norm{f}{1, A(R)}
    &\;\leq\;
    \Norm{\theta^{-1}}{1, A(R)}^{1/2}\,
    \Norm{f^2}{1, A(R), \theta}^{1/2}
    \nonumber\\[.5ex]
    &\;\leq\;
    c\, R^{-(m+d)/2}\,
    \Norm{\theta^{-1}}{1, B(R)}^{1/2}\,
    \norm{\eta^{m/2} f}{2}{B(R), \theta}
  \end{align}
  with $\eta\colon \boldsymbol{V} \to (0, \infty)$ defined in \eqref{e1-4}. Further, by using the H\"older inequality and the fact that $\rho_{*}/(\rho_{*}-2) = p$ we get
  \begin{align*}
    &\norm{f^{2}}{1}{B(R), \theta}
    \\[.5ex]
    &\mspace{36mu}\leq\;
    \bigl|B(R)\bigr|^{(\rho_{*}-2)/\rho_{*}}\,
    \Norm{\theta}{\rho_{*}/(\rho_{*}-2), B(R)}\,
    \norm{f^{2}}{\rho_{*}/2}{B(R)}
    \\[.5ex]
    &\mspace{36mu}\leq\;
    2\, \bigl|B(R)\bigr|^{1/p}\,
    \Norm{\theta}{p, B(R)}\,
    \Bigl(
      \Norm{f - (fg)_{B(R)}}{\rho_{*}, B(R)}^{2}
      + \bigl|B(R)\bigr|^{2/\rho_{*}} \bigl|(fg)_{B(R)}\bigr|^{2}
    \Bigr).
  \end{align*}
  %
  %
  Combining this estimate with \eqref{p2-1-2} and \eqref{eq:est_av} and using again \eqref{a2-2-1} yields, for every $R \geq R_{1}$,
  \begin{align} \label{eq:estL2f}
    &\norm{f^{2}}{1}{B(R), \theta}
    \nonumber\\[.5ex]
    &\;\leq\;
    c\,\cM\, R^{d/p}
    \Bigl(
      R^{d(2-\rho)/\rho + 2(1-d/d')}\, \cE(f)
      + R^{-m-d+2d/\rho_{*}}\,
      \norm{\eta^{m/2} f}{2}{\boldsymbol{V}, \theta}^2
    \Bigr)
    \nonumber\\
    &\;\leq\;
    c\, \cM\,
    \Bigl(
      R^{2}\, \cE(f)
      + R^{-m}\,
      \norm{\eta^{m/2} f}{2}{\boldsymbol{V}, \theta}^{2}
    \Bigr)
  \end{align}
  %
  %
  with $\cM$ as defined in \eqref{def:M}. In the last step we have used \eqref{a2-2-1} and the fact that
  \begin{align*}
    \frac{d}{p}
    + \frac{d(2-\rho)}{\rho}
    + \frac{2(d'-d)}{d'}
    \;=\;
    \frac{2d}{d'} + \frac{2(d'-d)}{d'}
    \;=\;
    2
    \qquad \text{and} \qquad
    \frac{1}{p} + \frac{2}{\rho_{*}} \;=\; 1,
  \end{align*}
  which can be verified directly by the definition of $\rho$ and $\rho_{*}$.

  Moreover, by the definition of $\eta$ we have for every $R \geq R_{1}$,
  \begin{align}\label{p2-1-5}
    \sum_{y \not\in B(R)} |f(y)|^{2} \theta(y)
    \;\leq\;
    R^{-m}\, \sum_{y \not\in B(R)} \eta(y)^{m} |f(y)|^{2} \theta(y)
    \;\leq\;
    R^{-m}\, \norm{\eta^{m/2} f}{2}{\boldsymbol{V}, \theta}^{2}.
  \end{align}
  This together with \eqref{eq:estL2f} yields that, for every $R \geq R_{1}$,
  \begin{align}\label{p2-1-4}
    \norm{f}{2}{\boldsymbol{V}, \theta}^{2}
    \;\leq\;
    c\, \cM\,
    \Bigl(
      R^{2}\, \cE(f)
      + R^{-m}\,
      \norm{\eta^{m/2} f}{2}{\boldsymbol{V}, \theta}^{2}
    \Bigr).
  \end{align}
  On the other hand, since $\eta(y)\ge 1$ by definition, for all $0 < R \leq R_{1}$,
  \begin{align}\label{p2-1-5a}
    \norm{f}{2}{\boldsymbol{V}\!,\, \theta}^{2}
    \;\leq\;
    \norm{\eta^{m/2} f}{2}{\boldsymbol{V}\!, \theta}
    \;\leq\;
    R_{1}^{m}\, R^{-m}\,
    \norm{\eta^{m/2} f}{2}{\boldsymbol{V}\!,\, \theta}^{2}.
  \end{align}
  Thus, by combining this estimate with \eqref{p2-1-4} we obtain
  \begin{align*}
    \norm{f}{2}{\boldsymbol{V}\!,\, \theta}^{2}
    \;\leq\;
    c\, R_{1}^{m}\, \cM\,
    \Bigl(
      R^{2}\, \cE(f)
      + R^{-m}\,
      \norm{\eta^{m/2} f}{2}{\boldsymbol{V}\!,\, \theta}^{2}
    \Bigr).
  \end{align*}
  Finally, by optimising the right-hand side of \eqref{p2-1-4} over $R > 0$, that is,
  \begin{align*}
    \inf_{R > 0}
    \Bigl(
      R^{2}\, \cE(f)
      + R^{-m}\,
      \norm{\eta^{m/2} f}{2}{\boldsymbol{V}\!,\, \theta}^{2}
    \Bigr)
    \;=\;
    2\, \cE(f)^{m/(m+2)}\,
    \norm{\eta^{m/2} f}{2}{\boldsymbol{V}\!,\, \theta}^{4/(m+2)},
  \end{align*}
  the desired assertion \eqref{p2-1-1} follows.
\end{proof}

\subsection{Interpolated anchored Nash inequality}
In this subsection we will show Theorem~\ref{thm:nash_int}. For the proof we will use the following lemma.
\begin{lemma}\label{l2-1}
  Let $a, a', b, b', c > 0$. Then, there exists constants $\alpha, \beta, \gamma \in (0, 1)$ with $\alpha + \beta + \gamma = 1$ such that, for any $A, B, C \in (0, \infty)$
  \begin{align}\label{l2-1-2}
    \inf_{r > 0, R > 0}
    \biggl(
      r^{a'} R^{a} A + r^{-b'}\biggl(\frac{R}{r}\biggr)^{\!\!b}\,B
      + R^{-c}\, C
    \biggr)
    \;\leq\;
    3\, A^{\alpha} B^{\beta} C^{\gamma}.
  \end{align}
  In particular, for $a' = 2 - d/q$, $a = d/q$, $b' = b = d$, $c = m$ with $d \geq 2$, $q > d/2$ and $m \geq 2$,
  \begin{align}\label{l2-1-1}
    \frac{d+2}{2}\,\beta + \frac{m+2}{2}\,\gamma \;=\; 1.
  \end{align}
\end{lemma}
%
%
\begin{proof}
  The proof is inspired by that of \cite[Lemma 2.10]{MO16}. To obtain an upper bound for the right-hand side of \eqref{l2-1-2}, we simply choose $r$ and $R$ so that the terms appearing in the sum are equal, that is,

  \begin{align*}
    r^{a'} R^{a} A
    \;=\;
    r^{-b'}\biggl(\frac{R}{r}\biggr)^{b} B
    \;=\;
    R^{-c} C.
  \end{align*}
  Note that the equality $r^{a'} R^{a} A = R^{-c} C$ implies $R = (C/r^{a'})^{1/(a+c)}$, while from the equality $r^{-b'} (R/r)^{b} B = R^{-c} C$ we deduce $R = (r^{b + b'} C/B)^{1/(b+c)}$. Combining those two equalities yields
  \begin{align}\label{l2-1-3}
    r
    \;=\;
    A^{-(b+c)/\sigma} B^{(a+c)/\sigma} C^{(b-a)/\sigma}
    \qquad \text{and} \qquad
    R
    \;=\;
    A^{-(b+b')/\sigma} B^{-a'/\sigma} C^{(a'+b'+b)/\sigma},
  \end{align}
  where $\sigma \ldef (b+b')c + a'c + a(b+b') + a'b$. Hence, by setting 
  \begin{align*}
    \alpha \;=\; \frac{(b+b')c}{\sigma},
    \qquad
    \beta \;=\; \frac{a'c}{\sigma},
    \qquad
    \gamma \;=\; \frac{a(b + b') + a'b}{\sigma}
  \end{align*}
  we clearly have that $\alpha, \beta, \gamma \in (0,1)$ and $\alpha + \beta + \gamma = 1$, which concludes the proof of \eqref{l2-1-3}. Moreover, in the particular case, where $a' = 2 - d/q$, $a = d/q$, $b = b' = d$ and $c = m$, it follows from an elementary computation that
  \begin{align*}
    \frac{d+2}{2}\, \beta + \frac{m+2}{2} \gamma
    &\;=\;
    1 - \alpha + \frac{d}{2} \beta + \frac{m}{2} \gamma
    \;=\;
    1 +
    \frac{%
      - 4bc + a'bc + 2abc + a'bc
    }{2 \sigma}
    \;=\;
    1.
  \end{align*}
\end{proof}
\begin{proof}[Proof of Theorem~\ref{thm:nash_int}]
  (i) First let us suppose that Assumption~\ref{a2-2} holds. Similarly as before, due to \eqref{a1-1-2} there exists $d' > d$ such that \eqref{p2-1-6} holds. Let $R_{1} \geq 2$, $C_{0} > 0$ be the constants associated with $d'$ in Assumption~\ref{a2-2}.

  We take $R \geq R_{1}$, $C_{0} R^{\alpha_{0}} \leq r \leq R$ and $x \in B(R)$ arbitrarily (where $\alpha_{0} \in (0, 1/2)$ is the same constant as in Assumption~\ref{a2-2}). Set $\rho = (1/d' + (p-1)/2p)^{-1}$, and note that $\rho < d'$ for any $p \in (1, \infty]$. In particular, \eqref{a2-2-3} implies that $\rho_{*} = 2p/(p-1)$. Then, by H\"older's inequality,
  %
  %
  \begin{align*}
    \Norm{f - (f)_{B(x, r), \theta}}{2, B(x, r), \theta}^{2}
    &\;=\;
    \inf_{a \in \mathbb{R}} \Norm{f - a}{2, B(x, r), \theta}^{2}
    \\[.5ex]
    &\;\leq\;
    \Norm{\theta}{\rho_{*}/(\rho_{*}-2), B(x, r)}\,
    \inf_{a \in \mathbb{R}} \Norm{f - a}{\rho_{*}, B(x, r)}^{2}
    \\[.5ex]
    &\;\leq\;
    \Norm{\theta}{\rho_{*}/(\rho_{*}-2), B(x, r)}\,
    \Norm{f - (f)_{B(x, r)}}{\rho_{*}, B(x, r)}^{2}.
  \end{align*}
  By using the volume regularity of balls, the Poincar\'{e}-Sobolev inequality as given in \eqref{a2-2-2} and H\"older's inequality, we obtain that
  \begin{align*}
    &\Norm{f - (f)_{B(x, r)}}{\rho_{*}, B(x, r)}^{2}
    \\[.5ex]
    &\mspace{36mu}\leq\;
    C_{\mathrm{S}}^{2}\, r^{2(1-d/d')}\, |B(x, r)|^{-2/\rho_{*}}\,
    \Biggl(
      \sum_{\substack{y, z \in B(x, \delta_{\mathrm{S}} r) \\ y \sim z}}
      \mspace{-15mu} \bigl| f(y) - f(z) \bigr|^{\rho}
    \Biggr)^{\!2/\rho}
    \\[.5ex]
    &\mspace{36mu}\leq\;
    c\, r^{d(2/d-1)}\,
    \Norm{\nu}{\rho/(\rho-2), B(x, \delta_{\mathrm{S}} n)}
    \sum_{\substack{y, z \in B(x, \delta_{\mathrm{S}} r) \\ y \sim z}}
    \mspace{-18mu}W(y, z) \bigl| f(y) - f(z) \bigr|^{2},
  \end{align*}
  where we used in the last step that $1/\rho_{*} = 1/\rho - 1/d'$. By combining the above estimates and using the fact that, in view of the definition of $\rho$ and $\rho_{*}$, the choice of $d'$ and \eqref{p2-1-6}, $\rho_{*} / (\rho_{*} - 2) = p$ and $\rho / (2-\rho) < q$, we get
  \begin{align}\label{eq:poincareTerm1}
    &\norm{f - (f)_{B(x, r), \theta}}{2}{B(x, r)}^{2}
    \nonumber\\[.5ex]
    &\mspace{36mu}\leq\;
    c\, r^{2}\, \Norm{\theta}{p, B(x, r)}\, \Norm{\nu}{q, B(x, r)}\,
    \sum_{\substack{y, z \in B(x, \delta_{\mathrm{S}} r) \\ y \sim z}}
    \mspace{-18mu}W(y, z) \bigl| f(y) - f(z) \bigr|^{2}.
  \end{align}
  Since $B(x,r) \subseteq B(2R)$, for any $1 \leq p_{1} \leq p_{2}$ we obtain by Jensen's inequality and the volume regularity as stated in \eqref{a2-2-1}
  \begin{align}\label{p2-2-5}
    \Norm{f}{p_{1}, B(x, r)}
    \;\leq\;
    \biggl( \frac{|B(2R)|}{|B(x, r)|} \biggr)^{\!\!1/p_{2}}\,
    \Norm{f}{p_{2},B(2R)}
    \;\leq\;
    c\, \biggl(\frac{R}{r}\biggr)^{d/p_{2}}\,
    \Norm{f}{p_{2}, B(2R)}.
  \end{align}
  Thus, from \eqref{eq:poincareTerm1} and the volume regularity \eqref{p2-2-5} we conclude that, for every $R \geq R_{1}$, $C_{0} R^{\alpha_{0}} \leq r \leq R$ and $x \in B(R)$,
  \begin{align}\label{p2-2-3a}
    &\norm{f - (f)_{B(x, r), \theta}}{2}{B(x, r)}^{2}
    \nonumber\\[.5ex]
    &\mspace{36mu}\leq\;
    c\, \cM\, r^{2} \biggl(\frac{R}{r}\biggr)^{\!d/p+d/q}
    \sum_{\substack{y, z \in B(x, \delta_{\mathrm{S}} r) \\ y \sim z}} \mspace{-15mu} W(y,z)\, \bigl|f(y)-f(z)\bigr|^{2},
  \end{align}
  %
  where $\cM$, still  defined as in \eqref{def:M}, is independent of $x$.
  
  On the other hand, by applying the Cauchy-Schwarz inequality and \eqref{p2-2-5},
  \begin{align*}
    1
    &\;\leq\;
    \Norm{\theta}{1, B(x, r)}\, \Norm{\theta^{-1}}{1, B(x, r)}
    \\[.5ex]
    &\;\leq\;
    c\, \biggl(\frac{R}{r}\biggr)^{\!d}\,
    \Norm{\theta}{1, B(x, r)}\, \Norm{\theta^{-1}}{1, B(2R)}
    \;\leq\;
    c\, \biggl(\frac{R}{r}\biggr)^{\!d} \cM\,
    \Norm{\theta}{1, B(x, r)},
  \end{align*}
  which implies that $\Norm{\theta}{1, B(x, r)} \geq c \cM^{-1} (R/r)^{-d}$ for every $R^{\alpha_{0}} \leq r \leq R$. Thus,
  \begin{align*}
    \bigl|(f)_{B(x,r), \theta}\bigr|^{2}\, \norm{\theta}{1}{B(x, r)}
    &\;=\;
    |B(x, r)|^{-1}\, \Norm{\theta}{1, B(x, r)}^{-1}\,
    \norm{f}{1}{B(x, r), \theta}^{2}
    \\[.5ex]
    &\;\leq\;
    c\, \mathcal{M}\, r^{-d} \biggl(\frac{R}{r}\biggr)^{\!d}\,
    \norm{f}{1}{B(x, r), \theta}^{2}.
  \end{align*}
  %
  %
  Using this and \eqref{p2-2-3a} we obtain that
  for every $R \geq R_{1}$, $C_{0} R^{\alpha_{0}} \leq r \leq R$ and $x \in B(R)$,
  \begin{align}\label{p2-2-3}
    \norm{f}{2}{B(x, r), \theta}^{2}
    &\;\leq\;
    2\, \norm{f - (f)_{B(x, r), \theta}}{2}{B(x, r), \theta}^{2}
    + 2\, \bigl|(f)_{B(x, r), \theta}\bigr|^{2}\, \norm{\theta}{1}{B(x, r)}
    \nonumber\\[.5ex]
    &\;\leq\;
    c\, \cM\,
    r^{2} \biggl(\frac{R}{r}\biggr)^{\!d/p+d/q}
    \sum_{\substack{y, z \in B(x, \delta_{\mathrm{S}} r) \\ y \sim z}} \mspace{-15mu} W(y,z)\, \bigl|f(y)-f(z)\bigr|^{2}
    \nonumber\\
    &\mspace{28mu}+
    c\, \cM\, r^{-d} \biggl(\frac{R}{r}\biggr)^{\!d}\,
    \norm{f}{1}{B(x, r), \theta}^{2}.
  \end{align}
  %

  Recall that
  \begin{align}\label{p2-2-6}
    \biggl(\sum_{i=1}^{N} |a_{i}|^{\zeta}\biggr)^{1/\zeta}
    \leq\;
    \sum_{i=1}^{N} |a_{i}|,
    \qquad
    \forall\, a_{i} \in \mathbb{R},\ N \in \mathbb{N},\ \zeta \geq 1.
  \end{align}
  Further, by the volume regularity condition \eqref{a2-2-1}, following the arguments in the proof of \cite[Theorem 2.1]{MO16}, there exists a positive integer $K_{0}$ such that for every $R \geq R_{1}$ and $C_{0} R^{\alpha_{0}} \leq r \leq R$, we can find a finite collection of balls $\{B(x_{i}, r)\}_{1 \leq i \leq N_{0}}$ with $x_{i} \in B(R)$ such that $B(R) \subset \bigcup_{i=1}^{N_{0}} B(x_{i}, r)$ and each point of $B(R)$ is covered by at most $K_{0}$ balls in $\{B(x_{i}, \delta_{\mathrm{S}} r)\}_{1 \leq i \leq N_{0}}$. In particular,
  \begin{align*}
    \sum_{i=1}^{N_{0}}
    \sum_{\substack{y, z \in B(x_{i}, \delta_{\mathrm{S}} r)\\ y \sim z}} \mspace{-18mu}W(y,z)\, \bigl|f(y) - f(z)\bigr|^{2}
    \;\leq\;
    c\, \cE(f),
  \end{align*}
  and by \eqref{p2-2-6},
  \begin{align*}
    \sum_{i=1}^{N_{0}}\;
    \norm{f}{1}{B(x_{i}, r), \theta}^{2}
    \;\leq\;
    \Biggl(
      \sum_{i=1}^{N_{0}}\; \norm{f}{1}{B(x_{i}, r), \theta}
    \Biggr)^{\!\!2}
    \;\leq\;
    c\, \norm{f}{1}{\boldsymbol{V}, \theta}^{2}.
  \end{align*}
  %
  %
  Hence, by using \eqref{p2-2-3}, we get for every $R \geq R_{1}$ and $C_{0} R^{\alpha_{0}} \leq r \leq R$,
  \begin{align*}
    \norm{f}{2}{B(R), \theta}^{2}
    &\;\leq\;
    \sum_{i=1}^{N_{0}}\, \norm{f}{2}{B(x_{i}, r), \theta}^{2}
    \\[.5ex]
    &\;\leq\;
    c\, \mathcal{M}
    \biggl(
      r^{2} \biggl(\frac{R}{r}\biggr)^{\!d/p+d/q} \cE(f)
      + r^{-d} \biggl(\frac{R}{r}\biggr)^{\!d/p+d/q}
      \norm{f}{1}{\boldsymbol{V}, \theta}^{2}
    \biggr).
  \end{align*}
  %
  Together with \eqref{p2-1-5} this estimate yields, for every $R \geq R_{1}$ and $C_{0} R^{\alpha_{0}} \leq r \leq R$,
  \begin{align}\label{p2-2-4}
    &\norm{f}{2}{\boldsymbol{V}, \theta}^{2}
    \;=\;
    \norm{f}{2}{B(R), \theta}^{2}
    + \sum_{x \not\in B(R)} |f(y)|^2\, \theta(y)
    \nonumber\\[.5ex]
    &\mspace{36mu}\leq\;
    c\, \cM
    \Biggl(
      r^2 \biggl(\frac{R}{r}\biggr)^{\!d/p+d/q} \cE(f)
      + r^{-d} \biggl(\frac{R}{r}\biggr)^{\!d}
      \norm{f}{1}{\boldsymbol{V}, \theta}^{2}
      + R^{-m} \norm{\eta^{m/2} f}{2}{\boldsymbol{V}, \theta}^{2}
    \Biggr).
  \end{align}
  %
   
  Further, by recalling that $1/p + 1/q < 2/d$, we have
  \begin{align*}
    r^{2} \biggl(\frac{R}{r}\biggr)^{\!d/p+d/q}
    \;=\;
    R^{2} \biggl(\frac{r}{R}\biggr)^{\!2-d/p-d/q}
    \;\geq\;
    R^{2},
    \qquad \forall\, R_{1} \leq R \leq r.
  \end{align*}
  Hence, by using this estimate together with \eqref{p2-1-4}, we immediately obtain that \eqref{p2-2-4} holds for every $r, R \in (0, \infty)$ with $R_{1} \leq R \leq r$.

  Finally, for every $r, R \in (0, \infty)$ with $R \geq R_{1}$ and $r \leq C_{0} R^{\alpha_{0}}$, we have
  \begin{align*}
    \norm{f}{2}{B(R), \theta}^{2}
    \;\leq\;
    \Bigl( \max_{y \in \boldsymbol{V}} \theta(y)^{-1} \Bigr)\,
    \norm{f}{1}{B(r), \theta}^{2}
    \;\leq\;
    c\, r^{-d} \biggl(\frac{R}{r}\biggr)^{\!d}
    \norm{f}{1}{B(R), \theta}^{2}.
  \end{align*}
  %
  Here we have used that $\inf_{y\in \boldsymbol{V}}\theta(y)\ge c$, \eqref{p2-2-6} and
  \begin{align*}
    r^{-d} \biggl(\frac{R}{r}\biggr)^{\!d}
    \;\geq\;
    C_{0}^{-2d} R^{d(1-2\alpha_{0})}
    \;\geq\;
    C_{0}^{-2d} R_{1}^{d(1-2\alpha_{0})}
    \;\geq\;
    C_{0}^{-2d}
  \end{align*}
  for all $r \leq C_{0} R^{\alpha_{0}}$ and $R \geq R_{1}$,
  where the last step exploits that $\alpha_{0} \in (0,1/2)$. Hence, \eqref{p2-2-4} holds for every $0 < r \leq C_{0} R_{1}$ and $R \geq R_{1}$.

  By combining all the above cases with \eqref{p2-1-5a} (for the case $0 \leq R \leq R_{1}$) we obtain that, for every $r, R \in (0, \infty)$,
  \begin{align}\label{p2-2-7}
    &\norm{f}{2}{\boldsymbol{V}, \theta}^{2}
    \nonumber\\[.5ex]
    &\mspace{18mu}\leq\;
    c\, R_{1}^{m}\, \cM
    \Biggl(
      r^{2} \biggl(\frac{R}{r}\biggr)^{\!d/p+d/q} \cE(f)
      + r^{-d} \biggl(\frac{R}{r}\biggr)^{\!d}
      \norm{f}{1}{\boldsymbol{V}, \theta}^{2}
      + R^{-m} \norm{\eta^{m/2} f}{2}{\boldsymbol{V}, \theta}^{2}
    \Biggr).
  \end{align}
  %
  %
  Thus, by applying Lemma~\ref{l2-1} with
  $A = \cE(f)$, $B = \norm{f}{1}{\boldsymbol{V}, \theta}^{2}$, $C = \norm{\eta^{m/2} f}{2}{\boldsymbol{V}, \theta}^2$ the assertion \eqref{p2-2-1} follows.
  \smallskip
  
  (ii) Now, suppose that Assumption~\ref{a2-2} holds for $R_{0} = R_{1} = 1$ and $\alpha_{0} = 0$. Note that under this assumption, \eqref{a2-2-1} and \eqref{a2-2-2} hold for every $R \geq 1$, $x \in B(R)$ and $1 \leq r \leq R$. Hence, by following the same procedure as in the proof of statement (i), \eqref{p2-2-7} holds for every $R$ and $r$ satisfying either $R \geq 2$, $r \geq 2$ or $R \leq 2$. Thus, it remains to verify \eqref{p2-2-7} for the regime $0 \leq r \leq 2\leq R$. Without the uniform lower bound on $\theta$, we have, for all $0 \leq r \leq 2 \leq R$,
  \begin{align*}
    \norm{f}{2}{B(R), \theta}^{2}
    \;\leq\;
    \norm{\theta^{-1}}{1}{B(R)}\, \norm{f}{1}{B(R), \theta}^{2}
    \;\leq\;
    c\, \mathcal{M}\, r^{-d} \biggl(\frac{R}{r}\biggr)^{\!d}
    \norm{f}{1}{B(R), \theta}^{2},
  \end{align*}
  %
  %
  where we applied the trivial bound $\theta(y')^{-1} \leq  \norm{\theta^{-1}}{1}{B(R)}$ for all $y' \in B(R)$ and \eqref{p2-2-6}. Hence, we obtain \eqref{p2-2-7} for all $r > 0$, $R > 0$ (with $R_{1} = 1$), which completes the proof of statement (ii).
\end{proof}
%
%
%
%
%

\section{Heat kernel estimates} \label{sec:hke}

\subsection{A general criterion for on-diagonal upper bounds}
In this subsection we first establish the following on-diagonal upper bound under a more abstract condition.

First we introduce some discrete calculus. In this paper we denote the collection of all oriented pairs of vertices on $\boldsymbol{V}$ by $\boldsymbol{\bar{E}}$, that is $\boldsymbol{\bar{E}} \ldef \{(y,z) : y,z \in \boldsymbol{V}\}$, where $y$ and $z$ stand for the starting vertex and terminal vertex for the pair $(y,z)$ respectively. Nothing of what will follow depends on the particular choice. For $f, g\colon \boldsymbol{\bar E} \rightarrow \bbR$ we write
\begin{align*}
  \scpr{f}{g}{\boldsymbol{\bar{E}}}
  \;\ldef\;
  \sum_{e \in \boldsymbol{\bar{E}}} f(e)\, g(e)
  \;=\;
  \sum_{x, y \in \boldsymbol{V}} f((x,y))\, g((x,y))
\end{align*}
for the $\ell^{2}$ inner product on $\boldsymbol{\bar{E}}$ with respect to the counting measure. For $f\colon \boldsymbol{V} \to \bbR$ and $F\colon \boldsymbol{\bar{E}} \to \bbR$ we define the operators $\nabla f\colon \boldsymbol{\bar{E}} \to \bbR$ and $\nabla^{*}F\colon \boldsymbol{V} \to \bbR$ by
\begin{align*}
  \nabla f(e)
  \;=\;
  \nabla f((e^+,e^-))
  \;\ldef\;
  f(e^+) - f(e^-),
  \quad
  \nabla^*F (x)
  \;\ldef
  \sum_{e: e^+ =\,x}\! F(e) \,-\! \sum_{e:e^-=\, x}\! F(e).
\end{align*}
Since $\sum_{e \in \boldsymbol{\bar{E}}} (\nabla f)(e) F(e) = \sum_{x \in \boldsymbol{V}} f(x) (\nabla^*F)(x)$ for all $f \in C_{c}(\boldsymbol{V})$ and $F \in C_{c}(\boldsymbol{E})$, $\nabla^*$ can be seen as the adjoint of $\nabla$.
\begin{prop}\label{p4-1}
  Suppose there exist $p, q \in (1, \infty]$ satisfying \eqref{a1-1-2} such that \eqref{a1-1-1} holds. Further, assume that the interpolated anchored Nash inequality \eqref{p2-2-1} holds, and there exist $m_{0} > d$ and $K_{1}, K_{2} \in (0,\infty)$ such that
  \begin{align}\label{l2-2-1}
    \mu_{2}(y)
    \;=\;
    \sum_{z \in \boldsymbol{V}} W(y,z)\, d(y,z)^{2}
    \;\leq\;
    K_{1}\, \theta(y),
    \qquad \forall\, y \in \boldsymbol{V},
  \end{align}
  and
  \begin{align}\label{l2-2-2}
    \sum_{\substack{z \in \boldsymbol{V} \\ d(y,z) > 2\eta(y)}} W(y,z)\, d(y,z)^{m_{0}}
    \;\leq\;
    K_{2}\,\eta^{m_{0}-2}(y)\, \theta(y),
    \qquad \forall\, y \in \boldsymbol{V}.
  \end{align}
  Then, there exists $C_{13} \in (0,\infty)$ such that, for all $t>0$,
  \begin{align}\label{t1-3-1}
    p(t,o,o)
    \;\leq\;
    C_{13}\, \Bigl(1 + C_{\mathrm{AN}}^{1/\beta}\Bigr)\,
    (1 + K_{1} + K_{2})^{m_{0}\gamma/(2\beta)}\, \cM^{1/\beta}\,
    \bigl(1 + \theta(o)^{-1}\bigr)^{(1-\alpha)/\beta}\, t^{-d/2},
  \end{align}
  where $\alpha, \beta, \gamma \in (0,1)$ are the same constants in Lemma \ref{l2-1} with $m = m_{0}$, $\cM$ is defined by \eqref{def:M}, and $C_{\mathrm{AN}}$ is the constant appearing in \eqref{p2-2-1}.
\end{prop}
Proposition~\ref{p4-1} will be proven at the end of this subsection. One key ingredient will be the following lemma.
\begin{lemma}\label{l2-2}
  Suppose that the assumptions of Proposition~\ref{p4-1} hold. Given $R \geq 1$ and $o \in \boldsymbol{V}$, for every $t > 0$ we define $u\colon\boldsymbol{V} \to [0, \infty)$ as the Dirichlet heat kernel associated with $(X_{t})_{t \geq 0}$ on $B(R)$, that is,
  \begin{align*}
    u_{t}(y)
    \;\ldef\;
    p_{B(R)}(t, o, y)
    \;=\;
    \frac{\Prob_{\!o}\bigl[X_{t} = y, t < \tau_{B(R)}\bigr]}{\theta(y)},
    \qquad t > 0,\ y \in \boldsymbol{V}.
  \end{align*}
  Then there exists $C_{14} \in (0, \infty)$ such that for every $t > 0$,
  \begin{align}\label{eq:energy2}
    \partial_{t} \norm{\eta^{m_{0}/2} u_{t}}{2}{\boldsymbol{V}\!,\, \theta}^2
    \;\leq\;
    C_{14}\, (K_{1} + K_{2})\,
    \norm{\eta^{m_{0}/2} u_{t}}{2}{\boldsymbol{V}\!,\, \theta}^{2(1-2/m_{0})}\,
    \norm{u_{t}}{2}{\boldsymbol{V}\!,\, \theta}^{4/m_{0}}.
  \end{align}
\end{lemma}
\begin{proof}
  Setting $v_{t}(y) \ldef \eta^{m_{0}/2-1}(y) u_{t}(y)$, we have that $v_{t}$ is supported on $B(R)$ and, hence, $v_{t} \in \ell^{2}(\boldsymbol{V}\!, \theta)$. Since by definition $\supp u_{t} \subset B(R)$ and $u$ solves $\partial_{t} u_{t}(y) = (\cL u_{t})(y)$ for $t > 0$ and $y \in B(R)$, using the integration by parts formula we get
  \begin{align} \label{eq:proof_energy}
    \frac{1}{2}\, \partial_{t} \norm{\eta^{m_{0}/2} u_{t}}{2}{\boldsymbol{V}\!,\, \theta}^{2}
    &\;=\;
    -\scpr{\nabla(\eta^{m_{0}} u_{t})}{W \nabla u_t}{\bar{\boldsymbol{E}}}
    \nonumber\\[.5ex]
    &\;=\;
    - \scpr{\nabla(\eta^{m_{0}/2+1} v_{t})}{W \nabla(\eta^{1-m_{0}/2} v_{t})}{\bar{\boldsymbol{E}}}
    \nonumber\\[.5ex]
    &\;=\;
    - \scpr{\nabla(\eta v_{t})}{W \nabla (\eta v_{t})}{\boldsymbol{\bar{E}}}
    + \scpr{\mathop{\mathrm{geo}}(v_{t})^{2}}{W (\nabla \eta)^{2}}{ \boldsymbol{\bar{E}}}
    \nonumber \\
    &\mspace{30mu}
    - \scpr{\mathop{\mathrm{geo}}(v_{t})^{2}}{W (\nabla\eta^{1+m_{0}/2}) (\nabla \eta^{1-m_{0}/2})}{\boldsymbol{\bar{E}}}
    \nonumber\\[.5ex]
    &\;\leq\;
    \scpr{\mathop{\mathrm{geo}}(v_{t})^{2}}{ W (\nabla \eta)^{2}}{\boldsymbol{\bar{E}}}
    \nonumber \\
    &\mspace{30mu}
    - \scpr{\mathop{\mathrm{geo}}(v_{t})^{2}}{W (\nabla\eta^{1+m_{0}/2}) (\nabla \eta^{1-m_{0}/2})}{\boldsymbol{\bar{E}}},
  \end{align}
  where $\mathop{\mathrm{geo}}(v_t)((x,y)) \ldef (v_{t}(x) v_{t}(y))^{1/2}$ denotes the geometric average of the function $v_{t}$ at the vertices $x$ and $y$.

  Recall that $\av{f}((x,y)) \ldef \frac{1}{2}(f(x) + f(y))$. By Lemma~\ref{lem:a1}, the definition of $v_{t}$ and Jensen's inequality we obtain
  \begin{align*}
    &\scpr{\mathop{\mathrm{geo}}(v_{t})^{2}}{W (\nabla \eta)^{2}}{\boldsymbol{\bar{E}}}
    - \scpr{\mathop{\mathrm{geo}}(v_{t})^{2}}{W (\nabla \eta^{1+m_{0}/2}) (\nabla \eta^{1-m_{0}/2})}{\boldsymbol{\bar{E}}}
    \\[.5ex]
    &\mspace{36mu}\leq\;
    \frac{m_{0}^{2}}{4}\,
    \scpr{\mathop{\mathrm{geo}}(v_{t})^{2}}{
      W \mathop{\mathrm{geo}}(\eta^{1-m_{0}/2})^{2} \av{\eta^{m_{0}/2 -1}}^{2} (\nabla \eta)^{2}
    }{\boldsymbol{\bar{E}}}
    \\[.5ex]
    &\mspace{36mu}=\;
    \frac{m_{0}^{2}}{4}\,
    \scpr{\mathop{\mathrm{geo}}(u_{t})^{2}}{W \av{\eta^{m_{0}/2 -1}}^{2} (\nabla \eta)^{2}}{\boldsymbol{\bar{E}}}
    \\[.5ex]
    &\mspace{36mu}\leq\;
    \frac{m_{0}^{2}}{4}\,
    \scpr{\mathop{\mathrm{geo}}(u_{t})^{2}}{W \av{\eta^{m_{0}-2}} (\nabla \eta)^{2}}{\boldsymbol{\bar{E}}}.
  \end{align*}
  Using the symmetry of $W$ and Young's inequality, we get
  \begin{align*}
    &\scpr{\mathop{\mathrm{geo}}(u_{t})^{2}}{W \av{\eta^{m_{0}-2}} (\nabla \eta)^{2}}{\boldsymbol{\bar{E}}}
    \\[.5ex]
    &\mspace{36mu}=\;
    \frac{1}{4}\,
    \sum_{y, z \in \boldsymbol{V}} u_{t}(y)u_{t}(z)\bigl(\eta^{m_{0}-2}(y) + \eta^{m_{0}-2}(z)\bigr)
    \bigl(\eta(y) - \eta(z)\bigr)^{2} W(y,z)
    \\
    &\mspace{36mu}\leq\;
    c\, \sum_{y, z \in \boldsymbol{V}} u_{t}^{2}(y) \eta^{m_{0}-2}(y) \bigl(\eta(y) - \eta(z)\bigr)^{2} W(y,z)
    \\
    &\mspace{72mu}+
    c\, \sum_{y, z \in \boldsymbol{V}} u_{t}^{2}(y) \eta^{m_{0}-2}(z) \bigl(\eta(y) - \eta(z)\bigr)^{2} W(y,z)
    \\[.5ex]
    &\mspace{36mu}\rdef\;
    I_{1}(t) + I_{2}(t).
  \end{align*}
  By the definition of $\mu_{2}$ and the fact that $|\eta(y) - \eta(z)| \leq d(y,z)$ for any $y, z \in \boldsymbol{\bar{E}}$,
  \begin{align*}
    I_{1}(t)
    &\;\leq\;
    c\, \sum_{y \in \boldsymbol{V}} u_{t}^{2}(y)\, \eta^{m_{0}-2}(y)\, \mu_{2}(y)
    \\[.5ex]
    &\;\leq\;
    c\, K_{1}\, \sum_{y \in \boldsymbol{V}} u_{t}^{2}(y)\, \eta^{m_{0}-2}(y)\, \theta(y)
    \;\leq\;
    c\, K_{1}\,
    \norm{\eta^{m_{0}/2} u_{t}}{2}{\boldsymbol{V}\!,\, \theta}^{2(1-2/m_{0})}\,
    \norm{u_{t}}{2}{\boldsymbol{V}\!,\, \theta}^{4/m_{0}},
  \end{align*}
  where the second inequality is due to \eqref{l2-2-1}. Using again that $|\eta(y) - \eta(z)| \leq d(y,z)$,
  \begin{align*}
    \eta^{m_{0}-2}(z)
    \;=\;
    \bigl(\eta(z) - \eta(y) + \eta(y)\bigr)^{m_{0}-2}
    &\;\leq\;
    c\,\bigl(d(y,z)^{m_{0}-2} + \eta^{m_{0}-2}(y)\bigr)
    \\[.5ex]
    &\;\leq\;
    c\, \eta^{m_{0}-2}(y) + c\, d(y,z)^{m_{0}-2}\, \indicator_{\{d(y,z) > 2\eta(y)\}}.
  \end{align*}
  Hence,
  \begin{align*}
    I_2(t)
    &\;\leq\;
    c\, \sum_{y \in \boldsymbol{V}} u_{t}^{2}(y)
    \sum_{z \in \boldsymbol{V}} W(y,z)\, d(y,z)^{m_{0}}\, \indicator_{\{d(y,z) > 2\eta(y)\}}
    \\
    &\mspace{28mu}
    + c\, \sum_{y \in \boldsymbol{V}} u_{t}^{2}(y)\, \eta^{m_{0}-2}(y)
    \sum_{z \in \boldsymbol{V}} W(y,z)\, d(y,z)^{2}
    \\[.5ex]
    &\;\rdef\;
    I_{21}(t) + I_{22}(t).
  \end{align*}
  Following the same arguments as above for $I_{1}(t)$ we get
  \begin{align*}
    I_{22}(t)
    \;\leq\;
    c\, K_{1}\, \norm{\eta^{m_{0}/2} u_{t}}{2}{\boldsymbol{V}\!,\, \theta}^{2(1-2/m_{0})}\,
    \norm{u_{t}}{2}{\boldsymbol{V}\!,\, \theta}^{4/m_{0}}.
  \end{align*}
  Further, applying condition \eqref{l2-2-2} and H\"older's inequality gives that
  \begin{align*}
    I_{21}(t)
    &\;\leq\;
    c\, K_{2}\, \sum_{y \in \boldsymbol{V}} u_{t}^{2}(y)\, \eta^{m_{0}-2}(y)\, \theta(y)
    \;\leq\;
    c\, K_{2}\, \norm{\eta^{m_{0}/2} u_{t}}{2}{\boldsymbol{V}\!,\, \theta}^{2(1-2/m_{0})}\,
    \norm{u_{t}}{2}{\boldsymbol{V}\!,\, \theta}^{4/m_{0}}.
  \end{align*}
  Combining the above estimates for $I_{1}(t)$ and $I_{2}(t)$ with \eqref{eq:proof_energy} yields \eqref{eq:energy2}.
\end{proof}
\begin{proof}[Proof of Proposition~\ref{p4-1}]
  Given Lemma~\ref{l2-2} the proof is similar to that of \cite[Theorem~3.2]{MO16}. For the reader's convenience we provide details here.

  For any $R \geq 1$, let $u_{t}(y) \ldef p_{B(R)}(t, o, y)$ be as before. Further, set
  \begin{align*}
    U_{t} \ldef \norm{u_{t}}{2}{\boldsymbol{V}\!,\, \theta}^{2}
    \qquad \text{and} \qquad
    V_{t} \ldef \norm{\eta^{m_{0}/2} u_{t}}{2}{\boldsymbol{V}\!,\, \theta}^{2},
  \end{align*}
  for any $t \in [0, \infty)$ and denote by $U_{t}'$ and $V_{t}'$ the derivatives with respect to $t$ of the mappings $t \mapsto U_{t}$ and $t \mapsto V_{t}$, respectively. Throughout this proof, all the constants $c$ appearing below will be independent of $R$.

  First note that by general properties of the heat kernel we have $U_{t}' = -2 \cE(u_{t})$, and by applying \eqref{eq:energy2} we get
  \begin{align}\label{t1-2-2}
    V_{t}'
    \;\leq\;
    c\, (K_{1} + K_{2})\, V_{t}^{1-2/m_{0}}\, U_{t}^{2/m_{0}},
    \qquad t > 0.
  \end{align}
  Now, we set $\Lambda_{t} \ldef \max\{1, \sup_{0 \leq s \leq t} s^{d/2} U_{s}\}$. Then, noting that $V_{0} = \theta(o)^{-1}$ and $m_{0} > d$, by integrating \eqref{t1-2-2}, we obtain
  \begin{align*}
    V_{t}^{2/m_{0}}
    &\;\leq\;
    V_{0}^{2/m_{0}} + c\, (K_{1} + K_{2})\, \Lambda_{t}^{2/m_{0}}\,
    \int_{0}^{t} s^{-d/m_{0}}\, \mathrm{d}s
    \\[.5ex]
    &\;\leq\;
    c\, (1 + K_{1} + K_{2})\,
    \Bigl(
      \theta(o)^{-2/m_{0}} + \Lambda_{t}^{2/m_{0}}\, t^{1 - d/m_{0}}
    \Bigr).
  \end{align*}
  Thus,
  \begin{align}\label{t1-2-3}
    V_{t}
    \;\leq\;
    c\, (1 + K_{1} + K_{2})^{m_{0}/2}\,
    \bigl(1 + \theta(o)^{-1}\bigr)\, \Lambda_{t}\, t^{(m_{0}-d)/2},
    \qquad \forall\, t \geq 1/4.
  \end{align}
  Further, by using the interpolated anchored Nash inequality \eqref{p2-2-1} together with the fact that $\norm{u_{t}}{1}{\boldsymbol{V}\!,\, \theta} \leq 1$ and $U_{t}' = -2\cE(u_{t})$, we have
  \begin{align*}
    U_{t}
    \;\leq\;
    C_{\mathrm{AN}}\, \cM\, \cE\big(u_{t}\big)^{\alpha}\, V_{t}^{\gamma}
    \;=\;
    -2^{-\alpha} C_{\mathrm{AN}}\, \cM\, \bigl(U_{t}'\bigr)^{\alpha}\, V_{t}^{\gamma}.
  \end{align*}
  Combining this with \eqref{t1-2-3} gives that, for all $t \geq 1/4$,
  \begin{align*}
    -U_{t}'
    \;\geq\;
    c\,
    \Bigl(
      C_{\mathrm{AN}}\, (1 + K_{1} + K_{2})^{m_{0}\gamma/2}
      \bigl(1 + \theta(o)^{-1}\bigr)^{\gamma} \cM
    \Bigr)^{\!-1/\alpha} U_{t}^{1/\alpha} \Lambda_{t}^{-\gamma/\alpha}\, t^{-\gamma(m_{0}-d)/(2\alpha)}\!.
  \end{align*}
  %
  Integrating this inequality and using that $\Lambda_{t}$ is increasing in $t$, we get for all $t \geq 1/2$,
  \begin{align}\label{t1-2-4}
    U_{t}^{1-1/\alpha}
    &\;\geq\;
    U_{t}^{1-1/\alpha} - U_{1/4}^{1-1/\alpha}
    \nonumber\\[.5ex]
    &\;\geq\;
    c\,
    \Bigl(
      C_{\mathrm{AN}}\, (1 + K_{1} + K_{2})^{m_{0}\gamma/2}
      \bigl(1 + \theta(o)^{-1}\bigr)^{\gamma} \cM
    \Bigr)^{\!-1/\alpha}
    \Lambda_{t}^{-\gamma/\alpha}\, t^{1-\gamma(m_{0}-d)/(2\alpha)}
    \nonumber\\[.5ex]
    &\;=\;
    c\,
    \Bigl(
      C_{\mathrm{AN}}\, (1 + K_{1} + K_{2})^{m_{0}\gamma/2}
      \bigl(1 + \theta(o)^{-1}\bigr)^{\gamma} \cM
    \Bigr)^{\!-1/\alpha}
    \Lambda_{t}^{-\gamma/\alpha}\, t^{d(1-\alpha)/(2\alpha)}.
  \end{align}
  %
  %
  In the last step we have used the fact
  \begin{align*}
    1 - \frac{\gamma(m_{0}-d)}{2\alpha}
    &\;=\;
    \frac{1}{\alpha}\biggl(1 - \beta - \gamma - \frac{\gamma(m_{0}-d)}{2}\biggr)
    \\
    &\;=\
    \frac{1}{\alpha}\biggl(\frac{d\beta}{2} + \frac{m_{0}\gamma}{2} - \frac{\gamma(m_{0}-d)}{2}\biggr)
    \;=\;
    \frac{d(1-\alpha)}{2\alpha}
    \;>\;
    0,
  \end{align*}
  which follows from \eqref{l2-1-1} and that $\alpha + \beta + \gamma = 1$. From \eqref{t1-2-4} and the fact that $t \mapsto \Lambda_{t}$ is increasing we deduce
  \begin{align*}
    \sup_{1/2 \leq s\leq t} s^{\frac{d}{2}} U_{s}
    \;\leq\;
    c \,
    \Bigl(
      C_{\mathrm{AN}}\, (1 + K_{1} + K_{2})^{m_{0}\gamma/2}
      \bigl(1 + \theta(o)^{-1}\bigr)^{\gamma} \cM
    \Bigr)^{\!1/(1-\alpha)}\,
    \Lambda_t^{\gamma/(1-\alpha)}.
  \end{align*}
  Combining this with the fact that $U_{t} = p_{B(R)}(2t, o, o) \leq \theta(o)^{-1}$ we get, for all $t > 0$,
  \begin{align*}
    \sup_{0 \leq s \leq t} s^{d/2} U_{s}
    \,\leq\,
    c\bigl(1 + C_{\mathrm{AN}}^{1/(1-\alpha)}\bigr)
    \Bigl(
      (1 + K_{1} + K_{2})^{m_{0}\gamma/2} \cM
    \Bigr)^{\!1/(1-\alpha)}\!
    \bigl(1 + \theta(o)^{-1}\bigr) \Lambda_{t}^{\gamma/(1-\alpha)}\!.
  \end{align*}
  By the definition of $\Lambda_{t}$ and the fact that $\beta > 0$, this implies
  \begin{align*}
    \Lambda_{t}
    &\leq\;
    c\bigl(1 + C_{\mathrm{AN}}^{1/\beta}\bigr)
    \Bigl(
      (1 + K_{1} + K_{2})^{m_{0}\gamma/2} \cM
    \Bigr)^{\!1/\beta}
    \bigl(1 + \theta(o)^{-1}\bigr)^{(1-\alpha)/\beta}.
  \end{align*}
  Hence, by using that $U_{t} \leq \Lambda_{t} t^{-d/2}$, we obtain, for all $t \geq 1/2$,
  \begin{align*}
    p_{B(R)}(2t,o,o)
    \;\leq\;
    c\bigl(1 + C_{\mathrm{AN}}^{1/\beta}\bigr)
    \Bigl(
      (1 + K_{1} + K_{2})^{m_{0}\gamma/2} \cM
    \Bigr)^{\!1/\beta}
    \bigl(1 + \theta(o)^{-1}\bigr)^{(1-\alpha)/\beta}\, t^{-d/2}.
  \end{align*}
  %
  %
  Recalling that all the appearing constants $c$ are independent of $R$, by taking $R \to \infty$ finishes the proof of \eqref{t1-3-1} for $t \geq 1$. Finally, since $p(t, o, o) \leq \theta(o)^{-1}$, it is easy follows that \eqref{t1-3-1} also holds for $t \in (0,1)$. This completes the proof of \eqref{t1-3-1}.
\end{proof}

\subsection{Proof of Theorem~\ref{t1-1}}
We start by observing that the stronger integrability condition \eqref{a1-4-1} is sufficient for \eqref{l2-2-2} to hold.
\begin{lemma}\label{l4-1}
  Suppose that \eqref{a2-2-1} holds and there exist $m_{0} > d$ and $p, q \in (1, \infty]$ satisfying
  \begin{align}\label{r1-1-2}
    \frac{1}{p} + \frac{1}{q} \;\leq\; \frac{m_{0}-2}{d},
  \end{align}
  such that
  \begin{align}\label{r1-1-1}
    \sup_{R \geq 1} \Norm{\theta^{-1}}{p, B(R)}
    + \sup_{R \geq 1}\Norm{\mu_{m_{0}}}{q, B(R)}
    \;<\;
    \infty.
  \end{align}
  Then, there exists $C_{15} \in (0, \infty)$ such that
  \begin{align*}
    \sum_{\substack{z \in \boldsymbol{V} \\ d(y,z) > 2\eta(y)}}
    \mspace{-18mu}W(y, z)\, d(y, z)^{m_{0}}
    \;\leq\;
    C_{15}\, R_{0}^{1/p+1/q} \widetilde{\cM}^2\, \eta^{m_{0}-2}(y)\, \theta(y),
    \qquad \forall\, y \in \boldsymbol{V},
  \end{align*}
  with $\widetilde{\cM} \ldef \sup_{R \geq 1} \bigl(\Norm{\theta^{-1}}{p, B(R)} + \Norm{\mu_{m_{0}}}{q, B(R)}\bigr)$ and $R_{0} \geq 1$ being the constant in Assumption \ref{a2-2}.
\end{lemma}
\begin{proof}
  Since $\sup_{y \in B(R)} \mu_{m_{0}}(y) \leq |B(R)|^{1/q} \Norm{\mu_{m_{0}}}{q, B(R)}$ for any $R \geq 1$, we have by the volume regularity condition \eqref{a2-2-1} (note that \eqref{a2-2-1} also implies that $|B(R)| \leq c \max\{R^{d}, R_{0}^{d}\}$ for every $R \geq 1$), that for all $y \in \boldsymbol{V}$,
  \begin{align}\label{r1-1-3}
    \mu_{m_{0}}(y)
    \;\leq\;
    \sup_{z \in B(\eta(y))} \mu_{m_{0}}(z)
    &\;\leq\;
    c\, \max\bigl\{\eta^{d/q}(y), R_{0}^{d/q}\bigr\}\,
    \sup_{R \geq 1}\, \Norm{\mu_{m_{0}}}{q, B(R)}
    \nonumber\\[.5ex]
    &\;\leq\;
    c\, R_{0}^{d/q}\, \eta^{d/q}(y)\,
    \sup_{R \geq 1}\, \Norm{\mu_{m_{0}}}{q, B(R)},
  \end{align}
  where we used in the last step that $\max\{\eta(y), R_{0}\} \leq R_{0} \eta(y)$ for every $y \in \boldsymbol{V}$ by the definition of $\eta$. Hence, by the definition of $\mu_{m_{0}}$ we get
  \begin{align*}
    \sum_{\substack{z \in \boldsymbol{V} \\ d(y,z) > 2\eta(y)}}
    \mspace{-18mu}W(y, z)\, d(y, z)^{m_{0}}
    \;\leq\;
    \mu_{m_{0}}(y)
    \;\le\;
    c\, \widetilde{\cM}\, R_{0}^{d/q}\, \eta^{d/q}(y).
  \end{align*}
  On the other hand, a similar argument as for \eqref{r1-1-3} gives
  \begin{align*}
    \theta^{-1}(y)
    \;\leq\;
    c\, \max\bigl\{\eta^{d/p}(y), R_{0}^{d/p}\bigr\}\,
    \sup_{R \geq 1}\, \Norm{\theta^{-1}}{p, B(R)}\, 
    \;\leq\;
    c\, \widetilde{\cM}\, R_{0}^{d/p}\, \eta^{d/p}(y),
    \qquad y \in \boldsymbol{V}.
  \end{align*}
  By combining the above estimates and using \eqref{r1-1-2} we obtain that, for all $y \in \boldsymbol{V}$,
  \begin{align*}
    \sum_{\substack{z \in \boldsymbol{V} \\ d(y,z) > 2\eta(y)}}
    \mspace{-18mu}W(y, z)\, d(y, z)^{m_{0}}
    &\;\leq\;
    c\, \widetilde{\cM}^{2}\, R_{0}^{d/p+d/q}\, \eta^{d/p+d/q}(y)\, \theta(y)
    \\[-1.5ex]
    &\;\leq\;
    c\, \widetilde{\cM}^{2}\, R_{0}^{d/p+d/q}\, \eta^{m_{0}-2}(y)\, \theta(y),
  \end{align*}
  which completes the proof.
\end{proof}

Now, we introduce a time-change $(Y_{t})_{t \geq 0}$ of $(X_{t})_{t \geq 0}$ which satisfies the conditions \eqref{l2-2-1} and \eqref{l2-2-2}. More precisely, let $(\cE, C_{c}(\boldsymbol{V}))$ be the quadratic form on $C_{c}(\boldsymbol{V})$ defined by \eqref{e1-3}. Given $m_{0} > d$, let $\pi_{m_{0}}$ be the measure on $\boldsymbol{V}$ given by
\begin{align}\label{e4-3}
  \pi_{m_{0}}(x)
  \;\ldef\;
  \max\bigl\{1, \mu_{m_{0}}(x), \theta(x)\bigr\},
  \qquad  x \in \boldsymbol{V},
\end{align}
where $\mu_{m_{0}}(x)$ is defined in \eqref{def:mu_nu}. As explained before, we can extend $(\cE, C_{c}(\boldsymbol{V}))$ to a regular Dirichlet form $(\cE_{m_{0}}, \cD(\cE_{m_{0}}))$ on $L^{2}(\boldsymbol{V}, \pi_{m_{0}})$ with reference measure $\pi_{m_{0}}$. Then, there exists a Hunt process $((Y_{t})_{t \geq 0}, \{\Prob_{\!x}\}_{x \in \boldsymbol{V}})$ associated with $(\cE_{m_{0}}, \cD(\cE_{m_{0}}))$ whose infinitesimal generator
$\cL_{Y}$ acts on $f \in C_{c}(\boldsymbol{V})$ as
\begin{align*}
  \bigl(\cL_{Y}f\bigr)(x)
  \;\ldef\;
  \frac{1}{\pi_{m_{0}}(x)}\, \sum_{y \in \boldsymbol{V}}\, W(x,y) \bigl(f(y) - f(x)\bigr).
\end{align*}
Thus, the process $(Y_t)_{t \geq 0}$ is reversible with respect to $\pi_{m_{0}}$. We write
\begin{align*}
  q(t, x, y)
  \;=\;
  \frac{\Prob_{\!x}\bigl[Y_{t} = y\bigr]}{\pi_{m_{0}}(y)}
\end{align*}
for the associated heat kernel. Furthermore, the two processes $(X_{t})_{t \geq 0}$ and $(Y_{t})_{t \geq 0}$ are time-changes of each other. More precisely, setting
\begin{align} \label{eq:def_A}
  A_{t}
  \;\ldef\;
  \int_{0}^{t} \frac{\theta(Y_{s})}{\pi_{m_{0}}(Y_{s})}\, \mathrm{d}s,
  \qquad t \geq 0,
\end{align}
we have that $(X_{t})_{t \geq 0} \overset{d}{=} (Y_{A_{t}^{-1}})_{t \geq 0}$, where $A_{t}^{-1} \ldef \inf\{s \geq 0 : A_{s} > t\}$, $t \geq 0$, denotes the right-continuous inverse of the additive functional $(A_{t})_{t \geq 0}$.
\begin{remark}
  Note that the inverse functional $A^{-1}$ is more explicitly given by
  \begin{align} \label{eq:Ainv}
    A_{t}^{-1}
    \;=\;
    \int_{0}^{t} \frac{\pi_{m_{0}}(X_{s})}{\theta(X_{s})}\, \mathrm{d}s,
    \qquad t \geq 0.
  \end{align}
  Indeed, by a change of variables $r = A_{s}^{-1}$ we obtain
  \begin{align*}
    \int_{0}^{t} \frac{\pi_{m_{0}}(X_{s})}{\theta(X_{s})}\, \mathrm{d}s
    \;=\;
    \int_{0}^{t} \frac{\pi_{m_{0}}(Y_{A^{-1}_{s}})}{\theta(Y_{A^{-1}_{s}})}\, \mathrm{d}s
    \;=\;
    \int_{0}^{A^{-1}_{t}}
      \frac{\pi_{m_{0}}(Y_{r})}{\theta(Y_{r})}\,
      \frac{\theta(Y_{r})}{\pi_{m_{0}}(Y_{r})}\,
    \mathrm{d}r
    \;=\;
    A_{t}^{-1},
  \end{align*}
  noting that $\mathrm{d}s = A'_{r}\, \mathrm{d}r = \theta(Y_{r})/ \pi_{m_{0}}(Y_{r})\, \mathrm{d}r$.
\end{remark}
\begin{proof}[Proof of Theorem \ref{t1-1}]
  (i) Recall that $\pi_{m_{0}}$, defined in \eqref{e4-3}, is the reference measure for the Dirichlet form $(\cE_{m_{0}}, \cD(\cE_{m_{0}}))$, so $\pi_{m_{0}}$ a reversible measure for the random walk $(Y_{t})_{t \geq 0}$. Let $\nu\colon \boldsymbol{V} \to [0, \infty)$ be still defined as in \eqref{def:mu_nu}. By the definition of $\mu_{m_{0}}$ and the fact that $m_{0} > d \geq 2$ we have $\mu_{2} \leq \mu_{m_{0}} \leq \pi_{m_{0}}$, which implies that \eqref{l2-2-1} holds for $\theta = \pi_{m_{0}}$ with $K_{1} = 1$. Moreover, Lemma~\ref{l4-1} gives that condition \eqref{l2-2-2} holds for $\theta = \pi_{m_{0}}$ with $K_{2} = C_{15} R_{0}^{1/p+1/q} \sup_{R \geq 1} \Norm{\mu_{m_{0}}}{p, B(R)}$ for $p \in \bigl(\max\{1,\frac{d}{m_{0}-2}\}, \infty\bigr]$ (here we used that $\pi_{m_{0}} \geq 1$). On the other hand, since $\pi_{m_{0}}^{-1} \leq \nu$, it follows directly from the integrability condition \eqref{a1-4-1} that condition \eqref{a1-1-1} holds for the choice $\theta = \pi_{m_{0}}$. Therefore, according to Theorem~\ref{thm:nash_int}, the interpolated Nash inequality \eqref{p2-2-1} holds for the reference measure $\theta = \pi_{m_{0}}$. Hence, the assumptions of Proposition~\ref{p4-1} hold for the choice $\theta = \pi_{m_{0}}$, which gives
  \begin{align}\label{t1-1-3}
    q(t, o, o)
    &\;=\;
    \frac{\Prob_{\!o}\bigl[Y_{t} = x\bigr]}{\pi_{m_{0}}(o)}
    \nonumber\\[.5ex]
    &\;\leq\;
    c\,
    \min\Bigl\{
      \pi_{m_{0}}(o)^{-1},
      R_{0}^{(1/p+1/q)m_{0}\gamma/(2\beta)}\, R_{1}^{m_{0}/\beta}\, \cM_{0}^{(1+m_{0}\gamma)/\beta}\,
      t^{-d/2}
    \Bigr\},
    \nonumber\\[.5ex]
    &\;\leq\;
    c\,
    \min\Bigl\{
      1,
      R_{0}^{(1/p+1/q)m_{0}\gamma/(2\beta)}\, R_{1}^{m_{0}/\beta}\, \cM_{0}^{(1+m_{0}\gamma)/\beta}\,
      t^{-d/2}
    \Bigr\}
  \end{align}
  for any $t > 0$, where we have used that $\cM \leq c \cM_{0}$ (here $\cM$ is associated with the reference measure $\pi_{m_{0}}$), which in turn follows from $\pi_{m_{0}}^{-1} \leq 1$.
  \smallskip

  We will now transfer this bound to the heat kernel $p(t,o,o)$ using the time change relation $X_{t} = Y_{A_{t}^{-1}}$, where $A_{t}$ is defined by \eqref{eq:def_A}. Note that the mapping $t \mapsto p(t,o,o)$ is non-increasing. Therefore,
  \begin{align}\label{t1-1-4}
    p(2t, o, o)
    \;\leq\;
    \frac{1}{t} \int_{t}^{2t} p(s,o,o)\, \mathrm{d}s
    &\;=\;
    \frac{1}{t\theta(o)}
    \int_{t}^{2t} \Prob_{\!o}\bigl[X_{s} = o\bigr]\, \mathrm{d}s
    \nonumber\\[.5ex]
    &\;=\;
    \frac{1}{t\theta(o)}
    \int_{t}^{2t} \Prob_{\!o}\bigl[Y_{A_{s}^{-1}} = o\bigr]\, \mathrm{d}s
    \nonumber\\[.5ex]
    &\;=\;
    \frac{1}{t\theta(o)}
    \Mean_{o}\Biggl[
      \int_{A_{t}^{-1}}^{A_{2t}^{-1}}
        \frac{\theta(Y_{r})}{\pi_{m_{0}}(Y_{r})}\,
        \indicator_{\{Y_{r} = o\}}\,
      \mathrm{d}r
    \Biggr].
  \end{align}
  Note that $\theta(x) \leq \pi_{m_{0}}(x)$ for every $x \in \boldsymbol{V}$ by the definition of $\pi_{m_{0}}$ in \eqref{e4-3}. Hence, $A_{t} \leq t$ for every $t > 0$, which implies that $A_{t}^{-1} \geq t$ for all $t > 0$. By applying those facts to \eqref{t1-1-4} we get
  \begin{align} \label{eq:bound1}
    p(2t, o, o)
    \;\leq\;
    \frac{1}{t\theta(o)}
    \int_{t}^{\infty} \Prob_{\!o}\bigl[Y_{r} = o, r \leq A_{2t}^{-1}\bigr]\, \mathrm{d}r.
  \end{align}
  Finally, by combining \eqref{eq:bound1} and \eqref{t1-1-3}, we obtain for every $t > 0$,
  \begin{align*}
    p(2t,o,o)
    &\;\leq\;
    \frac{1}{t\theta(o)}
    \int_{t}^{\infty} \Prob_{\!o}\bigl[Y_{r} = o \bigr]\, \mathrm{d}r
    \\[.5ex]
    &\;=\;
    \frac{\pi_{m_{0}}(o)} {t\theta(o)}
    \int_{t}^{\infty} q(r, o, o)\, \mathrm{d}r
    \\[.5ex]
    &\;\leq\;
    c\, \cM_{0}^{(1+m_{0}\gamma)/\beta}\, R_{0}^{(1/p+1/q)m_{0}\gamma/(2\beta)}\,
    R_{1}^{m_{0}/\beta}\,
    \frac{\pi_{m_{0}}(o)}{t \theta(o)}
    \int_{t}^{\infty} r^{-d/2}\, \mathrm{d}r
    \\[.5ex]
    &\;\leq\;
    c\, \cM_{0}^{(1+m_{0}\gamma)/\beta}\, R_{0}^{(1/p+1/q)m_{0}\gamma/(2\beta)}\,
    R_{1}^{m_{0}/\beta}\,
    \frac{\pi_{m_{0}}(o)}{\theta(o)}\, t^{-d/2},
  \end{align*}
  where we used that $d \geq 3$ in the final step.

  (ii) When $W(x,y) \neq 0$ only if $x \sim y$, we define a measure $\pi$ on $\boldsymbol{V}$ as $\pi(x) \ldef \max\bigl\{1, \mu(x), \theta(x)\bigr\}$, $x \in \boldsymbol{V}$. Then, it is easy to see that \eqref{l2-1-1} and \eqref{l2-2-2} hold for $\theta = \pi$ with $K_{1} = K_{2} = 1$. In particular, the condition \eqref{r1-1-2} is not needed to ensure \eqref{l2-2-2} for $\pi$, so the restriction $p \geq d/(m_{0}-2)$ is not required. Further, note that in the nearest neighbour setting $\mu = \mu_{m_{0}}$ for any $m_{0} \geq 0$. Then, by repeating the above time change procedure we obtain \eqref{t1-1-1a}.
\end{proof}

\section{Annealed heat kernel bounds for ergodic environments} \label{sec:rcm}
This section is devoted to the proof of Theorem~\ref{t1-3}.
\begin{proof}[Proof of Theorem~\ref{t1-3}]
  Recall that due to the symmetry of the heat kernel we have $p_{t}^{\omega}(x,y) \leq p^{\omega}_{t}(x, x)^{1/2} p^{\omega}_{t}(y, y)^{1/2}$ for all $x, y \in \bbZ^{d}$ and $t > 0$. Since $p_{t}^{\omega}(x, x) = p_{t}^{\tau_x \omega}(0, 0)$ for all $x \in \bbZ^{d}$, by shift-invariance and the Cauchy-Schwarz inequality it suffices to consider the case $x = y = 0$.
  \smallskip

  (i) Suppose $d \ge 3$. Recall that $\cM_{1}^{\omega}$ is defined in \eqref{t1-3-1aa}. By taking the expectation in \eqref{t1-3-1a} it suffices to show that $\cM_{1}^{\omega} \in L^{1}(\prob)$. Then, by H\"older's inequality and Young's inequality we have
  \begin{align*}
    \mean\bigl[\cM_{1}^{\omega}\bigr]
    \;\leq\;
    \mean\Bigl[ (\cM_{0}^{\omega})^{(1+m_{0}\gamma)q_{1}/\beta} \Bigr]^{1/q_{1}}\,
    \mean\biggl[
      \biggl(1 + \frac{\mu_{m_{0}}^{\omega}(0)}{\theta^{\omega}(0)}\biggr)^{\!q_{2}}
    \biggr]^{1/q_{2}},
  \end{align*}
  and
  \begin{align*}
    &\mean\Bigl[ (\cM_{0}^{\omega})^{(1+m_{0}\gamma)q_{1}/\beta} \Bigr]
    \\[.5ex]
    &\mspace{36mu}\le\;
    c\, \mean\biggl[
      \sup_{R \geq 1}\Norm{\theta^{\omega}}{p, B(R)}^{(1+m_{0}\gamma)p_{1} q_{1}/\beta}
      +
      \sup_{R \geq 1}\Norm{\mu_{m_{0}}^{\omega}}{p, B(R)}^{(1+m_{0}\gamma)p_{1} q_{1}/\beta}
    \biggr]
    \\
    &\mspace{58mu}+
    c\,
    \mean\biggl[
      \sup_{R \geq 1}\Norm{\nu^{\omega}}{q, B(R)}^{(1+m_{0}\gamma)p_{2} q_{1}/\beta}
    \biggr]
    \\[.5ex]
    &\mspace{36mu}\leq\;
    c\,
    \biggl(
      1
      +
      \mean\Bigl[
        \sup_{R \geq 1}\Norm{\theta^{\omega}}{\tilde{p}, B(R)}^{\tilde{p}}
        + \sup_{R \geq 1} \Norm{\mu_{m_{0}}^{\omega}}{\tilde{p}, B(R)}^{\tilde{p}}
      \Bigr]
    \biggr)
    \biggl(
      1
      +
      \mean\Bigl[
        \sup_{R \geq 1}\Norm{\nu^{\omega}}{\tilde{q}, B(R)}^{\tilde{q}}
      \Bigr]
    \biggr).
  \end{align*}
  This together with maximal ergodic theorem yields that $\mean\bigl[ \cM_{1}^{\omega}\bigr] < \infty$ as desired.
  \smallskip

  (ii) Now, we consider the case $d=2$. 
  Since by definition $p^{\omega}(t, 0, 0) \leq \theta^{\omega}(0)^{-1}$, it suffices to prove the desired conclusion for every $t \geq 2$.  Similarly as before, we set $\pi_{m_{0}}^{\omega}(y) \ldef 1 \vee \mu_{m_{0}}^{\omega}(y) \vee \theta^{\omega}(y)$, $y \in \bbZ^{d}$. Given $\omega \in \Omega$, denote by $((Y_{t}^{\omega})_{t \geq 0}, \{\mathbf{P}_{\!x}^{\omega}\}_{x \in \bbZ^{d}}) $ the time-changed process $Y_{t}^{\omega} = X_{A_{t}}^{\omega}$, $t \geq 0$, with the additive functional $(A_{t})_{t \geq 0}$ as defined in \eqref{eq:def_A} (adapted to the model in this theorem). Moreover, as explained in \cite{Bi11, BCKW21}, the $\Omega$-valued process $(\tau_{Y_{t}^{\omega}} \omega)_{t \geq 0}$, also known as the process of the environment as seen from the particle, is reversible and ergodic with respect to the probability measure $\prob_{\pi_{m_{0}}}$ given by
  \begin{align}\label{t1-3-7}
    \prob_{\pi_{m_{0}}}(\mathrm{d}\omega)
    \;\ldef\;
    \frac{\pi^{\omega}_{m_{0}}(0)}{\mean\bigl[\pi_{m_{0}}^{\omega}(0)\bigr]}\, \prob(\mathrm{d}\omega).
  \end{align}

  Since $\theta^{\omega}(y) \leq \pi_{m_{0}}^{\omega}(y)$ and therefore $A_{t}^{-1} \geq t$, we get from \eqref{t1-1-4} that
  \begin{align*}
    p^{\omega}(2t, 0, 0)
    &\;\leq\;
    \frac{1}{t\theta^{\omega}(0)}\,
    \bfE_{0}^{\omega}\biggl[
      \int_{t}^{A_{2t}^{-1}} \indicator_{\{Y_{r}^{\omega} = 0\}}\, \mathrm{d}r
    \biggr]
    \\[.5ex]
    &\;\leq\;
    \frac{1}{t\theta^{\omega}(0)}\,
    \bfE_{0}^{\omega}\biggl[
      \int_{t}^{8t} \indicator_{\{Y_{r}^{\omega} = 0\}}\, \mathrm{d}r
    \biggr]
    \\
    &\mspace{32mu}+
    \frac{1}{t\theta^{\omega}(0)}
    \sum_{k = \lfloor \frac{\log (8t)}{\log 2} \rfloor}^{\infty}
    \bfE_{0}^{\omega}\biggl[
      \biggl(
        \int_{t}^{2^{k+1}} \indicator_{\{Y_{r}^{\omega} = 0\}}\, \mathrm{d}r
      \biggr) \indicator_{\{2^{k} \leq A_{2t}^{-1} < 2^{k+1}\}}
    \biggr]
    \\
    &\;\rdef\;
    I_{0}^{\omega}(t)
    + \sum_{k = \lfloor \frac{\log (8t)}{\log 2} \rfloor}^{\infty} I_{k}^{\omega}(t).
  \end{align*}
  By the bound in \eqref{t1-1-3} we have that
  \begin{align*}
    I_{0}^{\omega}(t)
    &\;=\;
    \frac{\pi_{m_{0}}^{\omega}(0)}{t\theta^{\omega}(0)}
    \int_{t}^{8t} q^{\omega}(r, 0, 0)\, \mathrm{d}r
    \\[.5ex]
    &\;\leq\;
    c\, \bigl(\cM_{0}^{\omega}\bigr)^{(1+m_{0}\gamma)/\beta}\,
    \frac{\pi_{m_{0}}^{\omega}(0)}{t\theta^{\omega}(0)}
    \int_{t}^{8t} r^{-1}\, \mathrm{d}r
    \;\leq\;
    c\, \cM_{1}^{\omega}\, t^{-1}.
  \end{align*}
  Hence, since $\cM_{1}^{\omega} \in L^{1}(\prob)$ as shown in (i),
  \begin{align}\label{t1-3-9}
    \mean\bigl[I_{0}^{\omega}(t)\bigr]
    \;\leq\;
    c\, t^{-1},
    \qquad \forall\, t > 0.
  \end{align}

  We now turn to bound $I_{k}^{\omega}(t)$. First, by the Cauchy-Schwarz inequality,
  \begin{align}\label{t1-3-6}
    I_{k}^{\omega}(t)
    &\;\leq\;
    \frac{1}{t\theta^{\omega}(0)}\,
    \bfE_{0}^{\omega}\biggl[
      \biggl(
        \int_{t}^{2^{k+1}} \indicator_{\{Y_{r}^{\omega} = 0\}}\, \mathrm{d}r
      \biggr)^{\!\!2}
    \biggr]^{1/2}
    \bfP_{0}^{\omega}\bigl[ A_{2t}^{-1} \geq 2^{k} \bigr]^{1/2}.
  \end{align}
  Note that for every $k \geq \lfloor \frac{\log (8t)}{\log 2} \rfloor$ and $t \geq 2$,
  \begin{align} \label{eq:ChapKolm}
    &\bfE_{0}^{\omega}\biggl[
      \biggl(
        \int_{t}^{2^{k+1}} \indicator_{\{Y_{r}^{\omega} = 0\}}\, \mathrm{d}r
      \biggr)^{2}
    \biggr]
    \nonumber\\
    &\mspace{36mu}=\;
    2\,
    \bfE_{0}^{\omega}\biggl[
      \int_{t}^{2^{k+1}}
        \int_{t}^{r_2} \indicator_{\{Y_{r_{1}}^{\omega} = 0, Y_{r_{2}}^{\omega} = 0\}}\, \mathrm{d}r_{1}\,
      \mathrm{d}r_{2}
    \biggr]
    \nonumber\\
    &\mspace{36mu}=\;
    2\, \pi_{m_{0}}^{\omega}(0)^2
    \int_{t}^{2^{k+1}}
      \int_{t}^{r_2} q^{\omega}(r_{1}, 0, 0)\, q^{\omega}(r_{2} - r_{1}, 0, 0)\, \mathrm{d}r_{1}\,
    \mathrm{d}r_{2}
    \nonumber\\
    &\mspace{36mu}\leq\;
    c\, \bigl(\cM_{0}^{\omega}\bigr)^{2(1+m_{0}\gamma)/\beta}\, \pi_{m_{0}}^{\omega}(0)^{2}
    \int_{t}^{2^{k+1}}
      \int_{t}^{r_2} r_{1}^{-1}\, \min\bigl\{1, (r_2-r_1)^{-1} \bigr\}\, \mathrm{d}r_{1}\,
    \mathrm{d}r_{2}
    \nonumber\\[.5ex]
    &\mspace{36mu}\leq\;
    c\, \bigl(\cM_{0}^{\omega}\bigr)^{2(1+m_{0}\gamma)/\beta}\,
    \pi_{m_{0}}^{\omega}(0)^{2} \cdot (k+1)^{2},
  \end{align}
  where we used in the second step that by the Markov property
  \begin{align*}
    \bfP_{0}^{\omega}\bigl[ Y_{r_1}^{\omega} = 0, Y_{r_2}^{\omega} = 0 \bigr]
    \;=\;
    q^{\omega}(r_{1}, 0, 0)\, q^{\omega}(r_{2} - r_{1}, 0, 0)\, \pi_{m_{0}}^{\omega}(0)^{2}
  \end{align*}
  for all $t \leq r_{1} < r_{2} \leq 2^{k+1}$, and in the third step we used again \eqref{t1-1-3}. Next, whenever $A_{2^{k}} = \int_{0}^{2^{k}}\! \theta^{\omega}(Y_{s}^{\omega})/ \pi_{m_{0}}^{\omega}(Y_{s}^{\omega})\, \mathrm{d}s \leq 2t$, by the Cauchy-Schwarz inequality,
  \begin{align*}
    2^{k}
    \;\leq\;
    \biggl(
      \int_{0}^{2^{k}}
        \mspace{-6mu}\frac{\theta^{\omega}(Y_{s}^{\omega})}{ \pi_{m_{0}}^{\omega}(Y_{s}^{\omega})}\,
      \mathrm{d}s
    \biggr)^{\!\!1/2}
    \biggl(
      \int_{0}^{2^{k}}
        \mspace{-6mu}\frac{\pi_{m_{0}}^{\omega}(Y_{s}^{\omega})}{\theta^{\omega}(Y_{s}^{\omega})}\,
      \mathrm{d}s
    \biggr)^{\!\!1/2}
    \;\leq\;
    \sqrt{2t}
    \biggl(
      \int_{0}^{2^{k}}
        \mspace{-6mu}\frac{\pi_{m_{0}}^{\omega}(Y_{s}^{\omega})}{\theta^{\omega}(Y_{s}^{\omega})}\,
      \mathrm{d}s
    \biggr)^{\!\!1/2}.
  \end{align*}
  Therefore, by the definition of $A_{2t}^{-1}$,
  \begin{align*} 
    \mathbf{P}_{0}^{\omega}\Bigl[A_{2t}^{-1} \geq 2^k\Bigr]
    &\;=\;
    \mathbf{P}_{0}^{\omega}
    \Biggl[
      \int_{0}^{2^k}
        \mspace{-6mu}\frac{\theta^{\omega}(Y_{s}^{\omega})}{\pi_{m_{0}}^{\omega}(Y_{s}^{\omega})}
      \mathrm{d}s
      \leq 2t
    \Biggr]
    \\[.5ex]
    &\;\leq\;
    \mathbf{P}_{0}^{\omega}
    \Biggl[
      \int_{0}^{2^k}
      \mspace{-6mu}\frac{\pi_{m_{0}}^{\omega}(Y_{s}^{\omega})}{\theta^{\omega}(Y_{s}^{\omega})}
      \mathrm{d}s
      \geq \frac{2^{2k}}{2t}
    \Biggr]
    \;\leq\;
    \frac{2t}{2^{2k}}
    \int_{0}^{2^k}
      \mspace{-6mu}\bfE_{0}^{\omega}\biggl[
        \frac{\pi_{m_{0}}^{\omega}(Y_{s}^{\omega})}{\theta^{\omega}(Y_{s}^{\omega})}
      \biggr]\,
    \mathrm{d}s.
  \end{align*}
  By combining this estimate and \eqref{eq:ChapKolm} with \eqref{t1-3-6} we get
  \begin{align}\label{t1-3-8}
    \mean\bigl[I_{k}^{\omega}(t)\bigr]
    &\;\leq\;
    \frac{c(k+1)}{2^{k}\sqrt{t}}\,
    \mean\Biggl[
      \bigl(\cM_{0}^{\omega}\bigr)^{(1+m_{0}\gamma)/\beta}\,
      \frac{\pi_{m_{0}}^{\omega}(0)}{\theta^{\omega}(0)}
      \biggl(
        \int_{0}^{2^{k}}
          \mspace{-6mu}\bfE_{0}^{\omega}\biggl[
            \frac{\pi_{m_{0}}^{\omega}(Y_{s}^{\omega})}{\theta^{\omega}(Y_{s})}
          \biggr]
        \mathrm{d}s
      \biggr)^{\!\!1/2}
    \Biggr]
    \nonumber\\
    &\;\leq\;
    \frac{c(k+1)}{2^{k}\sqrt{t}}\,
    \mean\Bigl[\bigl(\cM_{1}^{\omega}\bigr)^{2}\Bigr]^{1/2}\,
    \Biggl(
      \int_{0}^{2^k}
        \mspace{-6mu}\mean\biggl[
          \bfE_{0}^{\omega}\biggl[
            \frac{\pi_{m_{0}}^{\omega}(Y_{s}^{\omega})}{\theta^{\omega}(Y_{s}^{\omega})}
          \biggr]
        \biggr]\,
      \mathrm{d}s
    \Biggr)^{\!\!1/2}.
  \end{align}
  Recall that $\prob_{\pi_{m_{0}}}$ given by \eqref{t1-3-7} is invariant for the environment process $(\tau_{Y_{t}^{\omega}} \omega)_{t \geq 0}$ and that $\pi_{m_{0}}^{\omega} \geq 1$. Hence, by using the stationarity of $\theta^{\omega}$ and $\pi_{m_{0}}^{\omega}$ as well as the moment condition \eqref{t1-3-3a1},
  \begin{align*}
    \mean\biggl[
      \bfE_{0}^{\omega}\biggl[
        \frac{\pi_{m_{0}}^{\omega}(Y_{s}^{\omega})}{\theta^{\omega}(Y_{s}^{\omega})}
      \biggr]
    \biggr]
    &\;=\;
    \mean\biggl[
      \bfE_{0}^{\omega}\biggl[
        \frac{\pi_{m_{0}}^{\tau_{Y_{s}}\omega}(0)}{\theta^{\tau_{Y_{s}}\omega}(0)}
      \biggr]
    \biggr]
    \\[.5ex]
    &\;\leq\;
    \mean\biggl[
      \bfE_{0}^{\omega}\biggl[
        \frac{\pi_{m_{0}}^{\tau_{Y_{s}}\omega}(0)}{\theta^{\tau_{Y_{s}}\omega}(0)}
      \biggr] \pi_{m_{0}}^{\omega}(0)
    \biggr]
    \;=\;
    \mean\biggl[
      \frac{\pi_{m_{0}}^{\omega}(0)^2}{\theta^{\omega}(0)}
    \biggr]
    \;<\;
    \infty.
  \end{align*}
  On the other hand, by the moment condition \eqref{t1-3-3a} and following the same arguments as in (i) above we can derive $\mean\bigl[(\cM_{1}^{\omega})^{2}\bigr] < \infty$. Hence, combining the above estimates with \eqref{t1-3-8} we get for every $k \geq \lfloor \log (8t)/\log 2 \rfloor$,
  \begin{align*}
    \mean\bigl[ I_{k}^{\omega}(t) \bigr]
    \;\leq\;
    c\, (k+1)\, 2^{-k/2}\, t^{-1/2}.
  \end{align*}
  This together with \eqref{t1-3-9} yields
  \begin{align*}
    \mean\bigl[ p^{\omega}(2t, 0, 0) \bigr]
    \;\leq\;
    c\,
    \Biggl(
      t^{-1}
      + t^{-1/2}\,
      \sum_{k = \lfloor \frac{\log (8t)}{\log 2} \rfloor}^{\infty}
      \mspace{-15mu}(k+1)\, 2^{-k/2}
    \Biggr)
    \;\leq\;
    c\, \log(t)\, t^{-1}.
  \end{align*}
  This completes the proof of \eqref{t1-3-4a}.
\end{proof}

\section{Nash inequalities and heat kernel bounds on percolation clusters}
We aim to apply Theorems~\ref{thm:nash_int} and \ref{t1-1} to the case $\boldsymbol{V} = \cC_{\infty}(\omega)$. For that purpose, the main task is to verify Assumption~\ref{a2-2} for the percolation cluster $\cC_{\infty}(\omega)$.
\begin{prop}\label{l5-2}
  Let $d \geq 2$ and suppose that Assumptions~\ref{a1-4} and \ref{a1-5} hold. Then, for any $d' > d$ and $\prob_{0}$-a.e.\ $\omega \in \Omega$, there exist a random constant $R_{1}(\omega) \geq 2$ and non-random constants $C_{16}, C_{17} \in (0, \infty)$ such that for all $R \geq R_{1}(\omega)$, $C_{16} R^{\alpha_{0}} \leq r \leq R$, $x \in Q^{\omega}(0, R)$ and any $1 < \rho < d'$ and $\rho_{*} > \rho$ satisfying \eqref{a2-2-3}, we have for all $f\colon \cC_{\infty}(\omega) \to \bbR$,
  \begin{align}\label{l5-2-1}
    \norm{f - (f)_{Q^{\omega}(x, r)}}{\rho_{*}}{Q^{\omega}(x, r)}
    \;\leq\;
    C_{17}\, r^{1-d/d'}\,
    \Biggl(
      \sum_{\substack{y, z \in Q^{\omega}(x, r) \\y \sim z}}
      \mspace{-15mu}\bigl|f(y) - f(z)\bigr|^{\rho}
    \Biggr)^{\!\!1/\rho}.
  \end{align}
\end{prop}
\begin{proof}
  \textsc{Step~1}. For any $f\colon \cC_{\infty}(\omega) \to \mathbb{R}$ let $m_{x,r}(f)$ denote the median of $f$ on $Q^{\omega}(x, r)$ under the normalized counting measure, that is,
  \begin{align*}
    \Bigl|
      \bigl\{
        y \in Q^{\omega}(x, r) : f(y) \leq m_{x,r}(f)
      \bigr\}
    \Bigr|
    \;\geq\;
    \frac{1}{2}\, \bigl|Q^{\omega}(x,r)\bigr|,
    \\
    \Bigl|
      \bigl\{
        y \in Q^{\omega}(x,r) : f(y) \geq m_{x,r}(f)
      \bigr\}
    \Bigr|
    \;\geq\;
    \frac{1}{2}\, \bigl|Q^{\omega}(x, r)\bigr|.
  \end{align*}
  We claim that to obtain \eqref{l5-2-1}, it suffices to prove that, for any $f\colon \cC_{\infty}(\omega) \to \bbR$,
  \begin{align}\label{l5-2-2}
    \norm{f - m_{x,r}(f)}{\rho_{*}}{Q^{\omega}(x, r)}
    \;\leq\;
    c\, r^{1-d/d'}\,
    \Biggl(
      \sum_{\substack{y, z \in Q^{\omega}(x, r) \\y \sim z}}
      \mspace{-15mu}\bigl|f(y) - f(z)\bigr|^{\rho}
    \Biggr)^{\!\!1/\rho}.
  \end{align}
  Indeed,
  \begin{align*}
    &\norm{f - (f)_{Q^{\omega}(x, r)}}{\rho_{*}}{Q^{\omega}(x, r)}
    \\[.5ex]
    &\mspace{36mu}\leq\;
    \norm{f - m_{x, r}(f)}{\rho_{*}}{Q^{\omega}(x, r)}
    \bigl|Q^{\omega}(x, r)\bigr|^{1/\rho_*}\,
    \bigl|m_{x, r}(f) - (f)_{Q^{\omega}(x, r)}\bigr|,
  \end{align*}
  %
  %
  and by H\"older's inequality
  \begin{align*}
    \bigl|m_{x, r}(f) - (f)_{Q^{\omega}(x,r)}\bigr|
    \;\leq\;
    |Q^{\omega}(x, r)|^{-1/\rho_{*}}\, \norm{f - m_{x, r}(f)}{\rho_{*}}{Q^{\omega}(x, r)}.
  \end{align*}
  %
  %
  Combining the above estimates with \eqref{l5-2-2} yields \eqref{l5-2-1}.
  \smallskip

  \textsc{Step~2.} We are now going to prove \eqref{l5-2-2}. First, we use the isoperimetric inequality \eqref{l5-1-2} and \cite[Lemma~3.3.13]{Sa96} to obtain that, for any $f\colon\cC_{\infty}(\omega) \to \bbR$,
  \begin{align}\label{l5-2-3}
    \norm{f - m_{x, r}(f)}{d_{*}}{Q^{\omega}(x, r)}
    \;\leq\;
    c\, r^{1-d/d'}\,
    \sum_{\substack{y, z \in Q^{\omega}(x, r) \\ y \sim z}}
    \mspace{-12mu}\bigl|f(y) - f(z)\bigr|,
  \end{align}
  where $d_{*} \ldef d'/(d'-1)$. Without loss of generality we may assume that $m_{x, r}(f) = 0$ (otherwise we replace $f$ by $f - m_{x, r}(f)$). Let $1 < \rho < d'$ and $\rho_{*} > \rho$ satisfy \eqref{a2-2-3}. Then, writing $f_{+}$ and $f_{-}$ for the positive and negative part of $f$, respectively, it is easy to verify that also $m_{x, r}(f_{+}^{\rho_{*}/d_{*}}) = 0$ and $m_{x,r}(f_{-}^{\rho_{*}/d_{*}}) = 0$. Applying \eqref{l5-2-3} to the function $f_{+}^{\rho_{*}/d_{*}}$ we get
  \begin{align}\label{l5-2-4}
    \norm{f_{+}}{\rho_{*}}{Q^{\omega}(x, r)}
    &\;=\;
    \norm{%
      f_{+}^{\rho_{*}/d_{*}} - m_{x, r}(f_{+}^{\rho_{*}/d_{*}})
    }{d_{*}}{Q^{\omega}(x, r)}^{d_{*}/\rho_{*}}
    \nonumber\\
    &\;\leq\;
    c\, r^{(1-d/d')d_{*}/\rho_{*}}\,
    \Biggl(
      \sum_{\substack{y, z \in Q^{\omega}(x, r) \\ y \sim z}}
      \mspace{-12mu}%
      \bigl|f_{+}(y)^{\rho_{*}/d_{*}} - f_{+}(z)^{\rho_{*}/d_{*}}\bigr|
    \Biggr)^{\!\!d_{*}/\rho_{*}}.
  \end{align}
  %
  %
  As a consequences of the mean-value theorem we have that
  \begin{align*}
    \bigl|
      f_{+}(y)^{\rho_{*}/d_{*}}
      - f_{+}(z)^{\rho_{*}/d_{*}}
    \bigr|
    \;\leq\;
    c\,
    \bigl| f_{+}(y) - f_{+}(z) \bigr|
    \bigl(
      f_{+}(y)^{\rho_{*}/d_{*}-1}
      + f_{+}(z)^{\rho_{*}/d_{*}-1}
    \bigr)
  \end{align*}
  which implies together with an application of  H\"older's inequality that
  \begin{align*}
    &\sum_{\substack{y, z \in Q^{\omega}(x,r) \\ y \sim z}}
    \mspace{-12mu}\bigl|
      f_{+}(y)^{\rho_{*}/d_{*}} - f_{+}(z)^{\rho_{*}/d_{*}}
    \bigr|
    \\
    &\mspace{36mu}\leq\;
    c\,
    \Biggl(
      \sum_{\substack{y, z \in Q^{\omega}(x,r) \\ y \sim z}}
      \mspace{-12mu}\bigl| f_{+}(y) - f_{+}(z) \bigr|^{\rho}
    \Biggr)^{\!\!1/\rho}\,
    \norm{f_{+}^{\rho_{*}/d_{*}-1}}{\rho/(\rho-1)}{Q^{\omega}(x, r)}
    \\
    &\mspace{36mu}\leq\;
    c\,
    \Biggl(
      \sum_{\substack{y, z \in Q^{\omega}(x,r) \\ y \sim z}}
      \mspace{-12mu}\bigl| f_{+}(y) - f_{+}(z) \bigr|^{\rho}
    \Biggr)^{\!\!1/\rho}\,
    \norm{f_{+}}{\rho_{*}}{Q^{\omega}(x, r)}^{\rho_{*}/d_{*}-1},
  \end{align*}
  %
  %
  where the last step is due to the fact $(\rho_{*}/d_{*}-1) \rho/(\rho-1) = \rho_{*}$, which can be verified by \eqref{a2-2-3}. This together with \eqref{l5-2-4} yields that
  \begin{align*}
    \norm{f_{+}}{\rho_{*}}{Q^{\omega}(x, r)}^{d_{*}/\rho_{*}}
    \;\leq\;
    c\, r^{(1-d/d')d_{*}/\rho_{*}}\,
    \Biggl(
      \sum_{\substack{y, z \in Q^{\omega}(x, r) \\ y \sim z}}
      \mspace{-12mu}\bigl|f_{+}(y) - f_{+}(z)\bigr|^{\rho}
    \Biggr)^{\!\!d_{*}/(\rho \rho_{*})},
  \end{align*}
  which can be rewritten as
  \begin{align}\label{l5-2-5}
    \norm{f_{+}}{\rho_{*}}{Q^{\omega}(x, r)}
    \;\leq\;
    c\, r^{1-d/d'}\,
    \Biggl(
      \sum_{\substack{y, z \in Q^{\omega}(x, r) \\ y \sim z}}
      \mspace{-12mu}\bigl|f_{+}(y) - f_{+}(z)\bigr|^{\rho}
    \Biggr)^{\!\!1/\rho}.
  \end{align}
  Noting that also $m_{x,r}(f_{-}^{\rho_{*}/d_{*}}) = 0$, we follow the same arguments as for $f_{+}$ to obtain
  \begin{align*}
    \norm{f_{-}}{\rho_{*}}{Q^{\omega}(x, r)}
    \;\leq\;
    c\, r^{1-d/d'}\,
    \Biggl(
      \sum_{\substack{y, z \in Q^{\omega}(x, r) \\ y \sim z}}
      \mspace{-12mu}\bigl|f_{-}(y) - f_{-}(z)\bigr|^{\rho}
    \Biggr)^{\!\!1/\rho}.
  \end{align*}
  By combining this estimate with \eqref{l5-2-5} and using the fact that
  \begin{align*}
    \bigl|f_{+}(y) - f_{+}(z)\bigr| + \bigl|f_{-}(y) - f_{-}(z)\bigr|
    \;\leq\;
    \bigl|f(y) - f(z)\bigr|,
    \qquad \forall\, y, z \in Q^{\omega}(x, r),
  \end{align*}
  we obtain \eqref{l5-2-2}, which implies \eqref{l5-2-1} by Step~1.
\end{proof}

\begin{proof}[Proof of Theorems~\ref{t1-4} and \ref{t1-5}]
  The results are immediate from Theorems~\ref{thm:nash_int} and \ref{t1-1} as soon as Assumption~\ref{a2-2} is verified. To that aim, note that under Assumption~\ref{a1-5}, for $\prob_{0}$-a.e.\ $\omega$, the volume growth condition \eqref{a2-2-1} holds on $\cC_{\infty}(\omega)$ due to \eqref{l5-1-1}. On the other hand, by Proposition~\ref{l5-2} we have that, $\prob_{0}$-almost surely, \eqref{a2-2-2} holds but with the balls $B^{\omega}(x, r)$ (induced by the standard graph distance $d^{\omega}$) replaced by $Q^{\omega}(x, r)$. However, by the volume growth condition on $Q^{\omega}(x, r)$ (see \eqref{l5-1-1}) as well as the fact that the $\ell^\infty$ distance $|x-y|_{\infty}$ and the graph distance $d^{\omega}$ are compartable on large scales (see \eqref{l5-1-3}), condition \eqref{a2-2-2} follows from this.
\end{proof}

\appendix
\section{Technical estimates}
\begin{lemma} \label{lem:a1}
  For any $a, b > 0$ and any $\alpha > 1$,
  \begin{align*}
    \bigl(a^{1+\alpha} - b^{1+\alpha}\bigr)
    \bigl(a^{1-\alpha} - b^{1-\alpha}\bigr)
    \;\geq\;
    (a - b)^{2} - \alpha^{2} a^{1-\alpha} b^{1-\alpha}
    \biggl(
      \frac{a^{\alpha-1} + b^{\alpha-1}}{2}
    \biggr)^{2} (a-b)^{2}.
  \end{align*}
\end{lemma}
\begin{proof}
  Note that
  \begin{align*}
    \bigl(a^{1+\alpha} - b^{1+\alpha}\bigr)
    \bigl(a^{1-\alpha} - b^{1-\alpha}\bigr)
    &\;=\;
    (a-b)^{2} - a^{1-\alpha} b^{1-\alpha}
    \bigl(a^\alpha - b^{\alpha}\bigr)^2
    \\[.5ex]
    &\;=\;
    (a-b)^{2} - \alpha^{2} a^{1-\alpha} b^{1-\alpha}\,
    \biggl(
      \int_{a}^{b} t^{\alpha-1}\, \mathrm{d}t
    \biggr)^{2}.
  \end{align*}
  Further, for any convex function $f:[a,b]\rightarrow \bbR$, we have
  \begin{align*}
    \int_{a}^{b} f(t)\, \mathrm{d}t
    \;\leq\;
    \frac{1}{2} \bigl(f(a) + f(b)\bigr) (a-b),
  \end{align*}
  which is due to the fact that the area under the graph of $f$ on the interval $[a, b]$ is bounded from above by the area under the secant connecting $a$ and $b$. Using this for the function $f(t) = t^{\alpha-1}$ and combining with the above, we obtain the claim.
\end{proof}

\subsection*{Acknowledgement}
S.A.\ has been supported by EPSRC grant EP/W022923/1.

\bibliographystyle{abbrv}
\bibliography{literature}

\end{document}